\documentclass[11pt,reqno]{amsart}
\pdfoutput=1

\usepackage[margin=27mm]{geometry}
\usepackage[T1]{fontenc}
\usepackage{lmodern}
\usepackage{microtype}
\usepackage{mathtools}
\usepackage{amssymb}
\usepackage{booktabs}
\usepackage{array}
\usepackage{xcolor}
\usepackage{tikz}
\usepackage{caption}
\usepackage[numbers,sort&compress]{natbib}
\usepackage{hyperref}
\usepackage[nameinlink,noabbrev]{cleveref}

\usetikzlibrary{decorations.markings,arrows.meta,calc}
\tikzset{ccw/.style={postaction={decorate,
 decoration={markings,mark=at position #1 with {\arrow{Stealth}}}}}}
\numberwithin{equation}{section}
\providecommand{\doi}[1]{doi:~\href{https://doi.org/#1}{\nolinkurl{#1}}}

\theoremstyle{plain}
\newtheorem{theorem}{Theorem}[section]
\newtheorem{lemma}[theorem]{Lemma}
\newtheorem{proposition}[theorem]{Proposition}
\newtheorem{corollary}[theorem]{Corollary}
\theoremstyle{definition}
\newtheorem{definition}[theorem]{Definition}
\theoremstyle{remark}
\newtheorem{remark}[theorem]{Remark}

\newcommand{\bv}[1]{\vec{#1}}
\newcommand{\bvp}[1]{\boldsymbol{#1}}

\newcommand{\region}[1]{\Omega_{#1}}
\newcommand{\interface}[1]{\Gamma_{#1}}
\newcommand{\normal}[1]{\bv \nu_{#1}}
\newcommand{\jump}[1]{[\![#1]\!]}
\newcommand{\curvature}[1]{\varkappa_{#1}}
\newcommand{\ball}[1]{\mathcal{B}_{#1}}
\newcommand{\tensionsym}{\sigma}
\newcommand{\tension}[1]{\tensionsym_{#1}}
\newcommand{\tensionVec}{\bvp{\tensionsym}}
\newcommand{\indexphasem}[1]{p_-(#1)}
\newcommand{\indexphasep}[1]{p_+(#1)}

\newcommand{\derivTheta}[1]{{#1}'}
\newcommand{\derivTime}[1]{\dot{#1}}
\newcommand{\rotation}[1]{\mathcal R_{#1}}

\newcommand{\R}{\mathbb R}

\newcommand{\chemical}[1]{w_{#1}}
\newcommand{\chemicalVec}[1]{\bvp w_{#1}}

\newcommand{\dd}{\mathrm d}

\newcommand{\compositesym}{\Sigma}
\newcommand{\Mcomp}[1]{M_{#1}^{\compositesym}}

\definecolor{figaccent}{RGB}{200,30,30}

\newif\ifredmarks
\redmarksfalse
\ifredmarks\else\colorlet{red}{black}\fi

\newcommand{\pstep}[2]{\textit{Step~#1 (#2).}\ }

\title[Thin-Layer Dichotomy in Three-Phase Mullins--Sekerka Flow]{Thin-Layer
 Dichotomy in Three-Phase Mullins--Sekerka Flow: Gap Dynamics and
 Composite Interfaces}
\author{Tokuhiro Eto}
\address{Universit\'{e} Claude Bernard Lyon 1, CNRS, Centrale Lyon,
 INSA Lyon, Universit\'{e} Jean Monnet, ICJ UMR5208,
 69622 Villeurbanne, France}
\email{eto@math.univ-lyon1.fr}
\date{}

\begin{document}

\begin{abstract}
We study the thin-layer limit of the Mullins--Sekerka flow in concentric circular geometries
at fixed wave number $k$, contrasting a two-phase annulus with a genuine
three-phase layer. A dichotomy emerges: for $k\ge2$ the rates blow up
like the inverse thickness when the same phase lies on both sides of the
layer, whereas between two different phases the matrix governing the
linearized Fourier mode extends analytically to zero thickness. Both
cases share the singular Dirichlet-to-Neumann factor of the thin annulus
and differ only in whether the Gibbs--Thomson law forces an order-one
potential difference across it. In the three-phase case the pure
thickness-modulation rate stays below the logarithmic growth rate of the
mean thickness, for every $k\ge1$ and all positive tensions, and, after
projection onto the mean displacement and the outer interface, the
collapsed pair moves as one interface carrying the sum of the two
tensions. All statements concern the instantaneous fixed-mode
linearization. The singularity is a property of the phase coupling, not
of the thin geometry.
\end{abstract}


\maketitle

\section{Introduction}\label{sec:intro}

A thin layer of one phase enclosed between two interfaces is a natural
regime in which the separation between the two interfaces may become
delicate.  If a phase occupies a layer of small thickness $\delta$,
local thinning may precede contact or a change of topology, and whether
a local reduction of the thickness is amplified or damped is the first
question to ask about such a layer.  For the Mullins--Sekerka flow
\cite{MS63}, this question is not local.
The normal velocity of an interface is determined by the jump
of the normal derivative of a harmonic chemical potential whose boundary
values are prescribed by the Gibbs--Thomson law on every interface;
the velocity at a point therefore depends on the whole configuration
through the diffusion field.  Two interfaces at distance $\delta$
interact through the field in the layer between them, and the
interaction can become strong as $\delta\downarrow0$: a harmonic
function that changes by an amount of order one across a layer of
thickness $\delta$ carries a flux of order $\delta^{-1}$ through it.
Whether such a change is actually imposed is, however, decided by the
conditions on the two interfaces and not by the geometry alone.
The purpose of this paper is to separate these two ingredients, thin
geometry and interface coupling, in the simplest setting in which the
separation can be carried out explicitly.

The two-phase Mullins--Sekerka flow shows that the singular interaction does occur.
Consider a thin annulus of a phase $C$ with the inner radius $R_1$ and outer radius $R_2$ embedded in a phase $A$, so that
the same phase lies on both sides of the layer. We represent this configuration as $A|C|A$
(Figure~\ref{fig:question}(a)). On the two circles bounding the annulus,
the Gibbs--Thomson law prescribes the same potential difference $w_A-w_C$.
The annulus is concave at its inner boundary and convex at its outer one,
so the prescribed values have opposite signs,
$+\tensionsym/R_1$ and $-\tensionsym/R_2$: one and the same difference must change
by about $2\tensionsym/R$ across the thickness $\delta$.
Therefore, the base state carries a flux of order $\tensionsym/(R\delta)$ through the layer,
and the factor $\delta^{-1}$ of the thin-annulus Dirichlet-to-Neumann map survives as a singular flux in the linearization.

This leads naturally to whether the singularity is a
consequence of the thin geometry alone.  Suppose that the layer
separates two distinct phases, $A|C|B$ with $A\ne B$
(Figure~\ref{fig:question}(b)).  The Gibbs--Thomson law now constrains
two different potential differences, $w_A-w_C$ on the inner circle and
$w_C-w_B$ on the outer one, and no single quantity is required to change
by an amount of order one across the layer; the third interface carries
$w_B-w_A$ and does not bound the layer.  The forcing of the two-phase
case is absent.  Nothing so far shows what replaces it: the
diffusion fields on the two sides of the layer remain coupled through
the thin annulus, and the annular Dirichlet-to-Neumann map still carries
the factor $\delta^{-1}$.  The question addressed in this paper is
whether the $\delta^{-1}$ singularity survives in a genuine three-phase
layer and, if it does not, what the limiting dynamics of the layer
thickness is.  The three-phase configuration thus enters not as one
instance of a multiphase flow but as the minimal model for this
question.

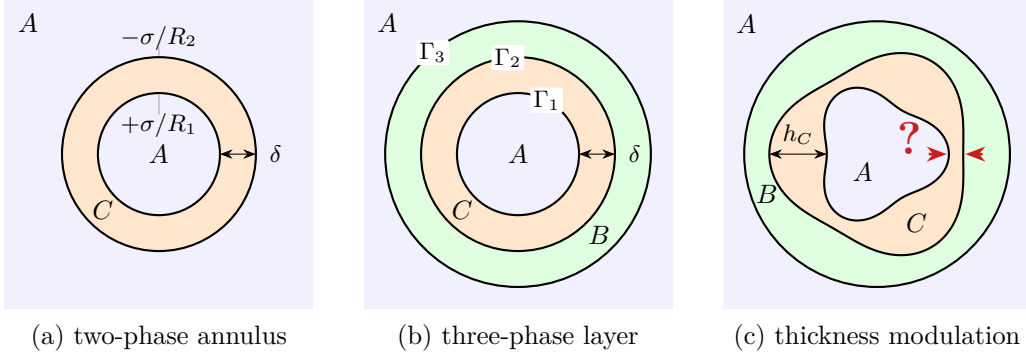
\begin{figure}[htb]
\centering
\begin{tikzpicture}[scale=0.95,>=Stealth,
  tag/.style={font=\footnotesize,fill=white,inner sep=1.2pt,rounded corners=1pt},
  phase/.style={font=\small},
  lead/.style={gray!70,line width=0.35pt},
  neck/.style={figaccent,line width=1pt}]
 \begin{scope}
  \fill[blue!6] (-2.15,-2.15) rectangle (2.15,2.15);
  \fill[orange!20] (0,0) circle (1.35);
  \fill[blue!6] (0,0) circle (0.85);
  \draw[thick] (0,0) circle (0.85);
  \draw[thick] (0,0) circle (1.35);
  \node[phase] at (0,0) {$A$};
  \node[phase] at (225:1.10) {$C$};
  \node[phase] at (-1.82,1.82) {$A$};
  \node[font=\footnotesize] at (0,0.40) {$+\tensionsym/R_1$};
  \draw[lead] (0,0.55) -- (0,0.83);
  \node[font=\footnotesize] at (0,1.62) {$-\tensionsym/R_2$};
  \draw[lead] (0,1.49) -- (0,1.37);
  \draw[<->] (0.85,0) -- (1.35,0);
  \node[font=\footnotesize,right] at (1.40,0) {$\delta$};
  \node[font=\small] at (0,-2.55) {(a) two-phase annulus};
 \end{scope}
 \begin{scope}[xshift=5.0cm]
  \fill[blue!6] (-2.15,-2.15) rectangle (2.15,2.15);
  \fill[green!12] (0,0) circle (1.85);
  \fill[orange!20] (0,0) circle (1.35);
  \fill[blue!6] (0,0) circle (0.85);
  \draw[thick] (0,0) circle (0.85);
  \draw[thick] (0,0) circle (1.35);
  \draw[thick] (0,0) circle (1.85);
  \node[phase] at (0,0) {$A$};
  \node[phase] at (225:1.10) {$C$};
  \node[phase] at (315:1.60) {$B$};
  \node[phase] at (-1.82,1.82) {$A$};
  \node[tag] at (62:0.85) {$\Gamma_1$};
  \node[tag] at (96:1.35) {$\Gamma_2$};
  \node[tag] at (130:1.85) {$\Gamma_3$};
  \draw[<->] (0.85,0) -- (1.35,0);
  \node[font=\footnotesize,right] at (1.40,0) {$\delta$};
  \node[font=\small] at (0,-2.55) {(b) three-phase layer};
 \end{scope}
 \begin{scope}[xshift=10.0cm]
  \fill[blue!6] (-2.15,-2.15) rectangle (2.15,2.15);
  \fill[green!12] (0,0) circle (1.85);
  \filldraw[even odd rule,fill=orange!20,draw=none]
   plot[domain=0:360,samples=181,smooth]
    ({(1.35-0.15*cos(3*\x))*cos(\x)},{(1.35-0.15*cos(3*\x))*sin(\x)}) -- cycle
   plot[domain=0:360,samples=181,smooth]
    ({(0.85+0.15*cos(3*\x))*cos(\x)},{(0.85+0.15*cos(3*\x))*sin(\x)}) -- cycle;
  \fill[blue!6] plot[domain=0:360,samples=181,smooth]
    ({(0.85+0.15*cos(3*\x))*cos(\x)},{(0.85+0.15*cos(3*\x))*sin(\x)}) -- cycle;
  \draw[thick] plot[domain=0:360,samples=181,smooth]
    ({(0.85+0.15*cos(3*\x))*cos(\x)},{(0.85+0.15*cos(3*\x))*sin(\x)});
  \draw[thick] plot[domain=0:360,samples=181,smooth]
    ({(1.35-0.15*cos(3*\x))*cos(\x)},{(1.35-0.15*cos(3*\x))*sin(\x)});
  \draw[thick] (0,0) circle (1.85);
  \node[phase] at (-0.20,-0.25) {$A$};
  \node[phase] at (300:1.10) {$C$};
  \node[phase] at (200:1.63) {$B$};
  \node[phase] at (-1.82,1.82) {$A$};
  \draw[<->] (-0.70,0) -- (-1.50,0);
  \node[font=\footnotesize] at (-1.10,0.28) {$h_C$};
  \draw[neck,->] (0.72,0) -- (0.97,0);
  \draw[neck,->] (1.47,0) -- (1.23,0);
  \node[figaccent,font=\LARGE\bfseries] at (0.45,0.22) {?};
  \node[font=\small] at (0,-2.55) {(c) thickness modulation};
 \end{scope}
\end{tikzpicture}
\caption{(a) A thin annulus of phase $C$ between two components of the
same phase $A$.  (b) A thin layer of $C$ between distinct phases $A$ and
$B$, with the three interfaces $\Gamma_1,\Gamma_2,\Gamma_3$.  (c) A
modulation of the thickness of the layer; $\delta$ is the
thickness of the layer, $h_C$ its mean thickness.}
\label{fig:question}
\end{figure}

We study this question for three concentric circles: the phase $A$ fills the
central disk and the unbounded exterior, the phase $C$ the inner annulus and
the phase $B$ the outer annulus, and we can represent the setting as $A|C|B|A$ in order of increasing radius
(Figure~\ref{fig:question}(b)).  Systems with three or more phases, in which interfaces of different
phase pairs interact through a common diffusion field, were derived and
analyzed in \cite{BGS98}; we use their formulation with one
chemical potential per phase.
The concentric evolution is an exact solution of the
conservative three-phase Mullins--Sekerka flow, i.e. the flow without reaction terms
and with no flux from infinity so that the two annular areas are conserved.
In this evolution, all three radii shrink, the central disk vanishes in finite time,
and the mean thickness $h_C=R_2-R_1$ of the $C$-layer increases
along the solution (Proposition~\ref{prop:base},
Corollary~\ref{cor:hC}).  The reference layer therefore widens while it
moves, and this changes what the relevant question is.  Whether a
modulation of the thickness grows or decays in absolute terms is only
part of the matter; for assessing whether thickness non-uniformity is
amplified, the relevant comparison is whether the modulation grows
relative to the mean thickness it modulates
(Figure~\ref{fig:question}(c)). We accordingly compare the
instantaneous rate of a thickness modulation with the logarithmic growth
rate of the mean thickness, and call the
difference the relative gap rate.

The rotational symmetry of the concentric state separates its
linearization into Fourier modes.  We perturb the three circles
radially, $r_j=R_j+a_j\cos k\theta$ with a fixed wave number $k\ge1$,
and linearize the flow in the amplitudes $\bvp a=(a_1,a_2,a_3)$; the
sine family behaves identically by rotation, and $k=0$ only moves the
state within the family of concentric solutions
(Remark~\ref{rem:kzero}).  Because the base state is radial, the
linearized problem does not mix wave numbers (Lemma~\ref{lem:modes}),
and within one mode the harmonic perturbations of the chemical
potentials are eliminated through the Dirichlet-to-Neumann maps of the
disk, of the two annuli and of the exterior.  What remains is a linear
system of ordinary differential equations for the three amplitudes,
\[
 \dot{\bvp a}=M_k(\bvp R)\,\bvp a,\qquad M_k(\bvp R)\in\R^{3\times3},
 \qquad \bvp R=(R_1,R_2,R_3),
\]
whose coefficient matrix we call the \emph{mode matrix}
(Proposition~\ref{prop:sixtrace}).  Thus $M_k$ is the restriction to
the mode-$k$ subspace of the linearization, about the concentric
solution, of the map from the shapes of the interfaces to their normal
velocities; the three circles are coupled because they share one
diffusion field.  It is not the Dirichlet-to-Neumann map itself: the
Gibbs--Thomson law converts the shape into boundary values of the
potentials, the Dirichlet-to-Neumann maps convert these boundary values
into fluxes, and the motion law converts the fluxes into velocities;
$M_k$ is the composite.  In particular $M_k$ depends on the tensions and
on which phases meet at each circle, whereas the Dirichlet-to-Neumann
maps depend on the geometry alone.  Since the radii move along the
concentric solution, $M_k(\bvp R(t))$ is a time-dependent matrix; every
statement of this paper concerns $M_k$ at a frozen instant.  The same
construction for a two-phase annulus, one phase occupying the annulus
and another its complement, gives a $2\times2$ mode matrix
$M^{(2)}_k$ (Proposition~\ref{prop:W5}).

For the inner pair we use the mean displacement and the thickness
modulation
\begin{equation}\label{eq:mg}
 m=\tfrac12(a_1+a_2),\qquad g=a_2-a_1 ,
\end{equation}
so that $m$ is the in-phase displacement of the layer and $g$ the gap,
its thickness modulation.  In the coordinates $(m,g,a_3)$ the
difference of the second and the first rows of $M_k$, the \emph{gap
row}, gives the rate of the thickness modulation,
\[
 \dot g=\beta_m\,m+\beta_g\,g+\beta_3\,a_3 .
\]
We call $\lambda_k:=\beta_g$ the \emph{pure-gap rate}, the
instantaneous logarithmic rate of a thickness modulation with
$m=a_3=0$, and $\Lambda_k:=\lambda_k-\frac{\dd}{\dd t}\log h_C$ the
\emph{relative gap rate}, the pure-gap rate measured against the
logarithmic growth rate \eqref{eq:hC} of the mean thickness
(Definition~\ref{def:gap}).  Finally, for radii $R<S$ and tensions
$\tension1,\tension2,\tension3$, the \emph{composite matrix}
$\Mcomp{k}$ is the two-phase mode matrix $M^{(2)}_k$ of the concentric
annulus with radii $(R,S)$, inner tension $\tension1+\tension2$ and
outer tension $\tension3$ (displayed in \eqref{eq:Mcomp-explicit}): it
describes a single interface carrying the sum of the two tensions of
the layer, inside the outer interface.

The thin-layer limit is the limit
$\delta=R_2-R_1\downarrow0$ at fixed inner radius $R$, fixed outer
radius $S$ and fixed $k$; we write $M_k^\delta$ for the mode matrix at
the frozen radii $(R,R+\delta,S)$.  The central question is what happens
to the $g$-dynamics and to $M_k^\delta$ in this limit.  Our main results
are the following.

\begin{theorem}[Thin-layer dichotomy and composite interface]
\label{thm:main}
Let $\tension{1},\tension{2},\tension{3}>0$ be the tensions of the three circles
$\interface{1},\interface{2},\interface{3}$, so that
$(\tension{1},\tension{2},\tension{3})=(\tension{AC},\tension{CB},\tension{BA})$,
let $R>0$, $\tau>1$, $S=\tau R$, and let $k\ge1$ be an integer.
For $\delta\in(0,S-R)$, let $M_k^\delta$ be the mode matrix of the concentric
solution with phase partition \eqref{eq:phase-partition} at the frozen
radii $(R_1,R_2,R_3)=(R,R+\delta,S)$, and set
\begin{equation}\label{eq:alphac}
 \alpha_c=\frac{(\tension{1}+\tension{2})/R+\tension{3}/S}{2\log(S/R)} .
\end{equation}
\begin{enumerate}
\item[(a)] $M_k^\delta$ extends to an analytic function of $\delta$ at
$\delta=0$; in particular it has a finite limit $M_k^0$.
\item[(b)] In the coordinates $(m,g,a_3)$ the limit has the form
\begin{equation}\label{eq:main-form}
 M_k^0=
 \begin{pmatrix}
 (\Mcomp{k})_{11} & c_1 & (\Mcomp{k})_{12}\\
 F_k & \lambda_k^0 & 0\\
 (\Mcomp{k})_{21} & c_3 & (\Mcomp{k})_{22}
 \end{pmatrix},
\end{equation}
where $\Mcomp{k}$ is the composite matrix with radii $(R,S)$ and
tensions $(\tension{1}+\tension{2},\tension{3})$, $c_1,c_3$ are real
constants depending on $k$, $R$, $S$ and the tensions, and
\begin{equation}\label{eq:main-gap}
 F_k=\frac{2k(k^2-1)(\tension{1}-\tension{2})}{3R^3},\qquad
 \lambda_k^0=-\frac{k(k^2-1)(\tension{1}+\tension{2})}{3R^3}
            +\frac{(3-k)\,\alpha_c}{3R^2}.
\end{equation}
\item[(c)] The pure-gap rate lies strictly below the limiting
logarithmic growth rate $\alpha_c/R^2$ of the mean thickness $R_2-R_1$,
that is, the relative gap rate is negative:
\begin{equation}\label{eq:main-relative}
 \Lambda_k^0:=\lambda_k^0-\frac{\alpha_c}{R^2}
 =-\frac{k}{3}\left(\frac{(k^2-1)(\tension{1}+\tension{2})}{R^3}
                     +\frac{\alpha_c}{R^2}\right)<0 .
\end{equation}
\item[(d)] If $\tension{1}=\tension{2}$ or $k=1$, then $F_k=0$, the subspace
$V=\{a_1=a_2\}$ is invariant under $M_k^0$, the restriction
$M_k^0|_V$ is $\Mcomp{k}$, and
\begin{equation}\label{eq:main-split}
 \operatorname{spec}M_k^0
 =\{\lambda_k^0\}\cup\operatorname{spec}\Mcomp{k},
\end{equation}
where $\operatorname{spec}A$ denotes the set of all eigenvalues of a square matrix $A$.
\end{enumerate}
\end{theorem}

The result for the three-phase case above is to be compared against the two-phase annulus
case of Figure~\ref{fig:question}(a), for which the same thin-layer
limit turns out to be singular.

\begin{theorem}[Two-phase thin-layer blow-up]
\label{thm:twophase}
Let a thin annulus of one phase with radii $R_1=R$ and $R_2=R+\delta$
be embedded in another phase, with a common tension $\tensionsym>0$ on
both circles, and let $M^{(2)}_k$ be its mode matrix.
\begin{enumerate}
\item[(a)] For each fixed integer $k\ge2$, as $\delta\downarrow0$, the
eigenvalues $\lambda_\pm$ of $M^{(2)}_k$ satisfy
\begin{equation}\label{eq:thin2}
 \lambda_\pm\sim\frac{\tensionsym}{R^2\delta}\,\mu_\pm(k),\qquad
 \mu_\pm(k)=\frac{-(k^2-1)\pm\sqrt{(k^2-1)(k^2+7)}}{2},
 \qquad \mu_+(k)>1,
\end{equation}
where $\lambda_\pm\sim G(\delta)$ means that $\lambda_\pm/G(\delta)\to 1$ as $\delta\downarrow 0$,
and the growing branch has a positive relative gap rate at leading
order,
\begin{equation}\label{eq:twophase-relative}
 \lambda_+-\frac{\dd}{\dd t}\log(R_2-R_1)
 =\frac{\tensionsym}{R^2\delta}\bigl(\mu_+(k)-1+O(\delta)\bigr).
\end{equation}
\item[(b)] For each fixed integer $k\ge2$, as $\delta\downarrow0$,
after a choice of sign both normalized eigenvectors tend to
$(1,1)^\top/\sqrt2$, that is, the distance of each of them to the pair
$\pm(1,1)^\top/\sqrt2$ tends to $0$, and the gap component of the
growing mode is $O(\delta)$, while the instantaneous pure-gap rate
stays finite,
\begin{equation}\label{eq:twophase-gap}
 \lambda_{\rm inst}:=-\tfrac12\bigl[(M^{(2)}_k\bvp v)_2-(M^{(2)}_k\bvp v)_1\bigr]
 \;\longrightarrow\;-\,\frac{\tensionsym\,k(k^2-1)}{2R^3},
 \qquad \bvp v=(1,-1)^\top .
\end{equation}
\item[(c)] For $k=1$, by contrast, the matrix $M^{(2)}_1$ is nonzero
and nilpotent for every $\delta>0$:
\begin{equation}\label{eq:twophase-k1}
 M^{(2)}_1=\frac{-\bar\alpha}{R_1R_2\sinh L}
 \begin{pmatrix}1&-1\\1&-1\end{pmatrix},
 \qquad
 \operatorname{spec}M^{(2)}_1=\{0\},
\end{equation}
where $L=\log(R_2/R_1)$ and $\bar\alpha=\tensionsym(1/R_1+1/R_2)/(2L)$
is the base coefficient of the two-phase annulus, $\dot R_j=-\bar\alpha/R_j$
\eqref{eq:base2}.
\end{enumerate}
\end{theorem}

Theorem~\ref{thm:twophase} is proved in Section~\ref{sec:twophase} and
Theorem~\ref{thm:main} in Section~\ref{sec:thinlimit}.  Let us
summarize the two theorems to three points.

\smallskip\noindent
\emph{The two-phase/three-phase dichotomy.}
The two-phase annulus has fixed-mode rates of order $\delta^{-1}$ at every mode
$k\ge2$; the three-phase layer has none.  In the two-phase annulus, the
singular rates belong to the coupled motion of the two circles rather
than to the thickness alone: both eigenvectors approach the in-phase
vector as $\delta\downarrow0$, and the instantaneous rate of a pure
thickness modulation stays bounded
(Theorem~\ref{thm:twophase}).  The two cases share the same
thin-annulus geometry and the same singular factor $\delta^{-1}$ in its
Dirichlet-to-Neumann map.  What differs is what multiplies that factor:
in the two-phase annulus the Gibbs--Thomson law forces a potential difference of order
one across the layer, while in the three-phase layer the interface
conditions make that difference of order $\delta$, and the product
stays bounded.  The singularity is a property of the phase coupling,
not of the thin geometry.  This is a statement at fixed $k$; the regime
$k\delta/R\asymp1$, in which the two circles of the layer decouple in
the Dirichlet-to-Neumann maps, is not treated.

\smallskip\noindent
\emph{Relative stabilization of the thickness.}
The rate of a pure thickness modulation is given in closed form.
It turns out to be positive for $k=1$, where the modulation is an
off-center displacement of the inner circle relative to the outer one,
and for $k=2$ when the outer annulus is thin
(Corollary~\ref{cor:T5}), while it is negative for $k\ge3$.  But it lies
below the logarithmic growth rate of the mean thickness for every $k$
and all positive tensions.  A modulation may grow in absolute size; in
the fixed-mode thin-layer limit it does not outgrow the layer.  This
is a statement about the instantaneous linearization, not a proof that
a three-phase film cannot rupture.  It concerns the pure-gap rate
alone: with unequal tensions the mean
displacement of the layer forces its thickness, so the inequality does
not by itself bound the normalized amplitude of an arbitrary
perturbation.

\smallskip\noindent
\emph{The composite interface.}  Projected onto the
mean displacement of the pair and the outer interface, the limit
dynamics is that of a two-interface problem in which the collapsed pair
is a single interface carrying the sum of its two tensions.  It is not
a two-phase model with one uniform tension between $A$ and $B$, and it
is not a wetting law.  When the two tensions agree, or when $k=1$, the
projection is exact: the in-phase subspace is invariant and the
spectrum splits into the pure-gap rate and the spectrum of the
composite problem.  For unequal tensions it is not, because an in-phase
displacement produces different Gibbs--Thomson perturbations on the two
sides of the layer; the exact splitting is a symmetry property.
The composite spectrum can contain a positive eigenvalue (for equal
tensions, $S=2R$ and $k=2$ it contains exactly one,
Remark~\ref{rem:positive-comp}); its eigenvector is an in-phase mode of
the outer shell with no thickness modulation, and for small $\delta$ it
persists with a gap component of order $\delta$.  The theorem is
therefore not a stability statement for the concentric configuration,
and a positive frozen eigenvalue is not in conflict with the negative
relative gap rate; a dynamical statement along the orbit requires a
separate non-autonomous estimate, which cannot rest on a uniform
negativity condition for the full frozen spectrum.

The closest Mullins--Sekerka comparison is the stability analysis by
Garcke et al. \cite{GLNS26} for a model of active droplets, in which a shell-type
configuration is sustained by reaction terms and is stationary.  The
present problem is conservative and has no reaction; its reference state
is an exact shrinking solution rather than an equilibrium; and the
object of study is the thin-interlayer limit rather than the stability
of a stationary shell.  Anjos and Li \cite{AnjosLi20} studied the thickness-dependent
coupling and the fingering of the two interfaces of an injection-driven
three-fluid annulus in a radial Hele--Shaw cell; in the in-phase
thin-annulus limit their equal-tension model reduces to a single
interface carrying the sum of the two tensions.
Liang et al. \cite{LiangEtAl11} found a related summed-tension limit for the in-phase
branch of thin concentric cylindrical shells in Stokes flow.
In both cases the reduction concerns two interfaces; in
the present three-interface problem the summed tension appears, for
arbitrary positive tensions, only after projection onto the mean
displacement and the outer interface.

We close the introduction with a brief review of classical results on the Mullins--Sekerka flow.
This mathematical model goes back to the linear stability analysis of Mullins and Sekerka \cite{MS63}
and to the ripening theory of Lifshitz--Slyozov and Wagner \cite{LS61,W61},
and it is the sharp-interface limit of the Cahn--Hilliard equation
\cite{CH58,Pego89,ABC94,S96}. For the two-phase case, well-posedness
\cite{Chen93,CHY96,EscherSimonett97,EMM24}, the stability of spheres \cite{ES98}
and of stationary solutions with boundary contact
\cite{ARW21,GR22}, and the long-time behavior in two dimensions
\cite{JMPS23} are established; see \cite{Garcke13,BDGP23} for surveys of
this modeling.

In the remainder, we mention literature of numerical schemes for the Mullins--Sekerka flow
together with its phase-field counterpart, the Cahn--Hilliard equation.
Phase-field methods evolve a
Cahn--Hilliard approximation of the flow \cite{Glasner03}; for the
mixed finite element method for the Cahn--Hilliard equation of
\cite{FengProhl04}, Feng and Prohl \cite{FengProhl05} proved that the
discrete solutions converge to the solution of the two-phase Hele--Shaw
problem as the interface thickness and the discretization parameters
tend to zero, provided that the Hele--Shaw problem admits a classical
solution; for the Stefan problem with the Gibbs--Thomson condition, a
level-set method couples the temperature field to an implicit interface
\cite{CMOS97};
boundary-integral methods discretize the
interface alone and represent the harmonic field by layer potentials
\cite{ZCH96,BCD95,CKT17,BB00} and have been used for large-scale
coarsening computations at finite area fraction
\cite{AV94,AkaiwaMeiron95}, building on boundary-integral and
gradient-flow frameworks for interfacial dynamics
\cite{HLS94,Mayer00,HLS01}, with hybrid and locally parametrized
variants for Hele--Shaw flow with surface tension \cite{HLOZ97,ZHH26},
with convergence analyses for interfacial
flows with surface tension \cite{CenicerosHou98,ALS17} and, through
recursively compressed inverse preconditioning, boundaries that contain
triple junctions \cite{HO08,H13}; front-tracking methods for curve
networks moving by curvature go back to \cite{BW95}, and parametric
finite element methods give unconditional stability and structure
preservation for geometric evolution equations
\cite{BGN07,BGN08,BGN08a,BGN10,BGN20,BZ21,R23}, for clusters of
immiscible phases \cite{N09} and for multiphase flow with triple
junctions \cite{GNZ26}, and, coupled to a
bulk field, for the two-phase Stefan and Mullins--Sekerka problems
\cite{BGN10-2} and for their multiphase versions with triple junctions
\cite{EGN24,EGN26,EGN26-2}; and the method of fundamental
solutions \cite{FK98,GC99}, in the charge simulation form \cite{M95},
gives a structure-preserving scheme for the one-phase Hele--Shaw problem
\cite{SY19} that extends to the two-phase Mullins--Sekerka flow and to
the contact-angle condition \cite{E24}, with structure-preserving
mesh-free and Galerkin boundary element versions for the multiphase flow
with triple junctions \cite{E26,E26-SGBEM}.  A further line discretizes
the energy rather than the equations, in the manner of the implicit time
discretization of mean curvature flow by Luckhaus and Sturzenhecker
\cite{LS95}: Otto \cite{Otto98} identified the one-phase
Mullins--Sekerka flow as the gradient flow of the perimeter with respect
to the Wasserstein metric, Chambolle and Laux \cite{CL21} proved that
the corresponding minimizing-movement scheme, up to a subsequence of the time-step size,
converges to a distributional solution without the perimeter-convergence hypothesis of
\cite{LS95}, and an oriented point-cloud varifold supplies
the perimeter term of that scheme \cite{EI26}; there the interface is carried by the points alone, and
the orientation detects the hidden boundary \cite{R05}, the part of the
interface that disappears when components merge.

\section{Three-phase conservative Mullins--Sekerka problem}
\label{sec:setting}

In this section, we consider a curve network composed of three concentric circles
which separate the whole space $\R^2$ into three phases.
For $r > 0$, let $\ball{r} := \{\bv x\in\R^2 \mid |\bv x|<r\}$, and fix $0<R_1<R_2<R_3$. Set
\begin{equation}\label{eq:phase-partition}
 \begin{aligned}
 \region{3}& := \ball{R_1}\cup(\R^2\setminus\overline{\ball{R_3}}),\\
 \region{1}& := \ball{R_2}\setminus\overline{\ball{R_1}},\\
 \region{2}& := \ball{R_3}\setminus\overline{\ball{R_2}},\qquad
 \interface{j}=\partial \ball{R_j}\quad(j=1,2,3).
 \end{aligned}
\end{equation}
In this case, the phase $3$ has two disconnected components;
the three circles separate the phase into pairs $(3,1)$, $(1,2)$, and $(2,3)$ in order of increasing radius.
Each circle $\interface{j}$ is oriented counter-clockwise, and the direction of its
normal $\normal{j}$ is taken outward so that the signed curvature of
the circles is $\curvature{j}=-1/R_j$.
Write the surface-tension coefficients of $\interface{j}$ as $\tension{j}$.
See Figure~\ref{fig:base-network} for the partition of $\R^2$ by the three concentric circles.

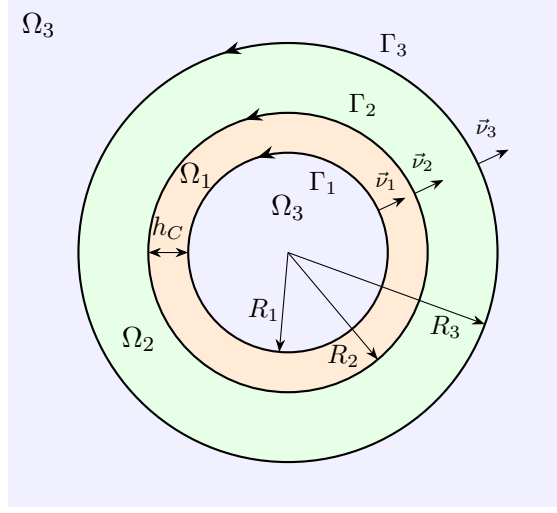
\begin{figure}[htb]
\centering
\begin{tikzpicture}[scale=1.32,>=Stealth]
 \fill[blue!6]    (-2.8,-2.55) rectangle (2.8,2.55);
 \fill[green!10]  (0,0) circle (2.1);
 \fill[orange!16] (0,0) circle (1.4);
 \fill[blue!6]    (0,0) circle (1.0);
 \draw[thick,ccw=0.30] (0,0) circle (1.0);
 \draw[thick,ccw=0.30] (0,0) circle (1.4);
 \draw[thick,ccw=0.30] (0,0) circle (2.1);
 \node[font=\small] at (63:0.78) {$\interface{1}$};
 \node[font=\small] at (63:1.66) {$\interface{2}$};
 \node[font=\small] at (63:2.34) {$\interface{3}$};
 \draw[->] (25:1.0) -- (25:1.30);
 \node[font=\scriptsize] at (36:1.22) {$\normal{1}$};
 \draw[->] (25:1.4) -- (25:1.72);
 \node[font=\scriptsize] at (34:1.63) {$\normal{2}$};
 \draw[->] (25:2.1) -- (25:2.44);
 \node[font=\scriptsize] at (32:2.35) {$\normal{3}$};
 \draw[<->] (180:1.0) -- (180:1.4);
 \node[font=\small,above] at (-1.2,0.03) {$h_C$};
 \draw[->] (0,0) -- (265:1.0) node[pos=0.58,left,font=\small,inner sep=1pt] {$R_1$};
 \draw[->] (0,0) -- (310:1.4) node[pos=0.83,below left,font=\small,inner sep=1pt] {$R_2$};
 \draw[->] (0,0) -- (340:2.1) node[pos=0.80,below,font=\small,inner sep=2pt] {$R_3$};
 \node at (90:0.45) {$\region{3}$};
 \node at (140:1.20) {$\region{1}$};
 \node at (210:1.73) {$\region{2}$};
 \node at (-2.50,2.28) {$\region{3}$};
\end{tikzpicture}
\caption{The concentric three-interface phase partition \eqref{eq:phase-partition}.}
\label{fig:base-network}
\end{figure}

We use the formulation of the multi-phase Mullins--Sekerka flow with
one chemical potential per phase \cite{BGS98}.
We write $\chemicalVec{} = (\chemical{1},\chemical{2},\chemical{3})^\top$ for the chemical
potentials. Then, they need to satisfy
\begin{subequations}\label{eq:fields}
  \begin{align}
 \Delta \chemical{p} = 0 & \qquad\text{in}\quad \region{\ell},\quad p = 1,2,3,\ \ell = 1,2,3,\label{eq:fields-L}\\
 \chemical{1}+\chemical{2}+\chemical{3}=0 &\qquad\text{in}\quad \R^2.\label{eq:fields-S}
  \end{align}
\end{subequations}
Meanwhile, the interface conditions on $\Gamma_j$ read
\begin{subequations}\label{eq:interfaces}
\begin{align}
 \jump{\chemical{p}}&=0, && \qquad p = 1,2,3,\label{eq:interfaces-C}\\
 \chemicalVec{}\cdot\jump{\bvp\chi}&=\tension{j}\curvature{j}, &&\qquad j =1,2,3, \label{eq:interfaces-GT}\\
 \nabla w_p(x) &=O\left(\frac{1}{|x|^2}\right), && \qquad\text{as}\quad |x|\to\infty, \label{eq:interfaced-FX}\\
 V_j\jump{\chi_p}&= - \jump{\nabla \chemical{p}}\cdot\normal{j}, &&\qquad p = 1,2,3,\ j=1,2,3,\label{eq:interfaces-ML}
\end{align}
\end{subequations}
where $\chi_p:\R^2\to \{0,1\}$ denotes the characteristic function of the phase $p$, and $\bvp \chi = (\chi_1,\chi_2,\chi_3)^\top$;
$\jump{q}$ denotes the gap (or jump) of a quantity $q$ across the boundary $\interface{j}$ in the direction $\normal{j}$.
Namely, it is defined by $\jump{q}(\bv x) := \lim_{\varepsilon\downarrow 0}\{q(\bv x + \varepsilon\normal{j}) - q(\bv x - \varepsilon\normal{j})\}$ for $\bv x\in\interface{j}$.

We now introduce the exact solution to the evolution law \eqref{eq:interfaces-ML}.

For each $j$, let $\indexphasem{j}$ be the phase on the inner side of $\interface{j}$
(the side opposite to $\normal{j}$), $\indexphasep{j}$ the phase on the outer side,
and $q(j)$ the remaining phase.
Noting that $\jump{\chi_{\indexphasem{j}}}=-1$, $\jump{\chi_{\indexphasep{j}}}=1$ and
$\jump{\chi_{q(j)}}=0$ on $\interface{j}$,
the motion law \eqref{eq:interfaces-ML} states that
\begin{equation}\label{eq:ML-components}
 V_j=\jump{\nabla\chemical{\indexphasem{j}}}\cdot\normal{j}
 =-\jump{\nabla\chemical{\indexphasep{j}}}\cdot\normal{j},
 \qquad
 \jump{\nabla\chemical{q(j)}}\cdot\normal{j}=0 ,
\end{equation}
The curvature in \eqref{eq:interfaces-GT} is taken in the direction of $\normal{j}$ and is negative for convex
curves, so that on the reference circles $\curvature{j}=-1/R_j$ and
\eqref{eq:interfaces-GT} reads
$\chemical{\indexphasem{j}}-\chemical{\indexphasep{j}}=\tension{j}/R_j$.

Since $\chemical{3}=-\chemical{1}-\chemical{2}$ by \eqref{eq:fields-S},
it is convenient to work with the reduced vector $\bvp W:=(\chemical{1},\chemical{2})^\top$
and to encode the phases by
\begin{equation}\label{eq:phase-vectors}
  \begin{aligned}
 \bvp c_1 & =(1,0)^\top,\qquad \bvp c_2=(0,1)^\top,\qquad \bvp c_3=(-1,-1)^\top,\\
 \chemical{p} &= \bvp c_p^\top\bvp W,
 \qquad\bvp d_j:=\bvp c_{\indexphasem{j}}-\bvp c_{\indexphasep{j}} ,
  \end{aligned}
\end{equation}
so that the Gibbs--Thomson condition \eqref{eq:interfaces-GT} reads
$\bvp d_j^\top\bvp W=-\tension{j}\curvature{j}$, i.e.\
$\bvp d_j^\top\bvp W=\tension{j}/R_j$ on the reference circles $\interface{j}$.
We summarize the relationship between the curve index and introduced indexes as follows:

\begin{equation}\label{eq:iface-table}
\begin{array}{c|ccc|c}
j & \indexphasem{j} & \indexphasep{j} & q(j) &\bvp d_j^\top\\ \hline
 1 & 3 & 1 & 2 & (-2,-1)\\
 2 & 1 & 2 & 3 & (1,-1)\\
 3 & 2 & 3 & 1 & (1,2)
\end{array}
\qquad\quad \bvp d_1+\bvp d_2+\bvp d_3=(0,0)^\top.
\end{equation}

We are now in a position to derive the rigorous evolution law of each circle $\interface{j}$ under the system \eqref{eq:fields} and \eqref{eq:interfaces}.
\begin{proposition}\label{prop:base}
For every $0<R_1<R_2<R_3$, the problem \eqref{eq:fields} and
\eqref{eq:interfaces} has exactly one radially
symmetric solution of the form $\chemical{p}(\bv x)=A_p+B_p\log |\bv x|$ in each of the two annuli $\region{1}$ and $\region{2}$,
with $B_p=0$ in $\region{3}$.
The velocities of $\interface{j}$ are governed by
\begin{equation}\label{eq:base-flow}
 \dot R_j=V_j=-\frac{\alpha}{R_j}\quad(j=1,2,3),
 \qquad
 \alpha=\frac{\tension{1}/R_1+\tension{2}/R_2+\tension{3}/R_3}{2\log(R_3/R_1)}>0 .
\end{equation}
Therefore, all three radii decrease.
Moreover, the two annular areas are conserved, that is,
\begin{equation}\label{eq:areas}
 R_2^2(t)-R_1^2(t)\equiv R_2^2 - R_1^2,\qquad\text{and}\quad R_3^2(t)-R_2^2(t)\equiv R_3^2 - R_2^2 
\end{equation}
hold. The central disk vanishes at a finite time $t_\star$.
\end{proposition}

\begin{proof}
In the annuli $\region{1}$ and $\region{2}$, we set
$\bvp W(\bv x)=\bvp K_1+\bvp b_1\log|\bv x|$ and $\bvp W(\bv x)=\bvp K_2+\bvp b_2\log|\bv x|$ with constant
vectors $\bvp K_1,\,\bvp K_2,\,\bvp b_1,\,\bvp b_2\in\R^2$ for $\bv x\in\R^2\setminus\{\bv 0\}$;
in the disk part of $\region{3}$, $\bvp W$ is equal to a constant $\bvp X$ by regularity, and in the exterior part of $\region{3}$,
it is equal to a constant by \eqref{eq:interfaced-FX}.
The no-flux relation in \eqref{eq:ML-components} on
$\interface{1}$, $\interface{2}$, $\interface{3}$ gives
$\bvp c_2^\top \bvp b_1=0$, $\bvp c_3^\top(\bvp b_2-\bvp b_1)=0$ and $\bvp c_1^\top\bvp b_2=0$, hence we see that
$\bvp b_1=(\alpha,0)^\top$ and $\bvp b_2=(0,\alpha)^\top$ hold for some $\alpha\in\R$.
The continuity condition \eqref{eq:interfaces-C} implies that
\begin{equation}
  \label{eq:base-1}
  \bvp W = \begin{cases}
    \bvp X &\qquad\text{on}\quad \interface{1},\\
    \bvp X+\alpha \bvp c_1\log\frac{R_2}{R_1} & \qquad\text{on}\quad \interface{2},\\
    \bvp X+\alpha \bvp c_1\log\frac{R_2}{R_1}+\alpha \bvp c_2\log\frac{R_3}{R_2} &\qquad\text{on}\quad \interface{3}.
  \end{cases}
\end{equation}
Noting that
$\bvp d_j^\top\bvp W=\tension{j}/R_j$ on $\interface{j}$ and
$d_1+d_2+d_3=0$ together with $d_2^\top c_1=d_3^\top c_1=1$ and $d_3^\top c_2=2$,
multiplying each equation by $\bvp d_j^\top$ and adding the all conditions of \eqref{eq:base-1} eliminates $X$ and gives
\[
 2\alpha\log(R_3/R_1)
 =\frac{\tension{1}}{R_1}+\frac{\tension{2}}{R_2}+\frac{\tension{3}}{R_3},
\]
which shows the second equation of \eqref{eq:base-flow}.

Meanwhile, a direct calculation shows that
\begin{equation}\label{eq:base-jumps}
  \begin{aligned}
 J_1&:=\jump{\nabla\bvp W}_{\interface{1}}\normal{1}=\frac{\alpha}{R_1}\bvp c_1,\\
 J_2&:=\jump{\nabla\bvp W}_{\interface{2}}\normal{2}=-\frac{\alpha}{R_2}(\bvp c_1- \bvp c_2),\\
 J_3&:=\jump{\nabla\bvp W}_{\interface{3}}\normal{3}=-\frac{\alpha}{R_3}\bvp c_2 .
  \end{aligned}
\end{equation}
Taking $\indexphasem{j}$-component in \eqref{eq:ML-components},
we observe that $V_1=\bvp c_3^\top J_1$, $V_2=\bvp c_1^\top J_2$, $V_3=\bvp c_2^\top J_3$,
all equal to $-\alpha/R_j$, and the first equation of \eqref{eq:base-flow} follows.

It remains to prove \eqref{eq:areas} and the finite-time extinction.
By \eqref{eq:base-flow}, we have
\[
\frac{\dd}{\dd t}R_j^2=2R_j\dot R_j=-2\alpha\qquad(j=1,2,3),
\]
so the three squared radii decrease at the same rate, and their
differences are constant in time, and thus \eqref{eq:areas} holds.
In particular, the ordering $R_1<R_2<R_3$ is preserved, and $R_2$ and
$R_3$ are determined by $R_1$ through
\begin{equation}\label{eq:radii-slaved}
 R_2(t)=\sqrt{R_1(t)^2+A_1},\qquad
 R_3(t)=\sqrt{R_1(t)^2+A_1+A_2},
\end{equation}
where $A_1:=R_2(0)^2-R_1(0)^2>0$ and $A_2:=R_3(0)^2-R_2(0)^2>0$ are
the two conserved annular areas divided by $\pi$.
Substituting \eqref{eq:radii-slaved} into \eqref{eq:base-flow}, we
regard $\alpha$ as a function of $R_1$ alone, and the evolution of the
innermost circle $\interface{1}$ is governed by the scalar equation
$\dot R_1=-\alpha(R_1)/R_1<0$.
Since the numerator of $\alpha$ in \eqref{eq:base-flow} is a sum of
three positive terms, we may drop the terms containing $\tension{2}$
and $\tension{3}$ to obtain the lower bound
\[
 \alpha(R_1)\ \ge\ \frac{\tension{1}}{2R_1\log(R_3/R_1)} .
\]
Separating variables in $\dot R_1=-\alpha(R_1)/R_1$ and writing
$u=R_1$, the time $t_\star$ at which $R_1$ reaches zero is
\[
 t_\star=\int_0^{R_1(0)}\frac{u\,\dd u}{\alpha(u)}
 \le\frac{2}{\tension{1}}\int_0^{R_1(0)}u^2\log\frac{R_3(u)}{u}\,\dd u<\infty ,
\]
where $R_3(u)=\sqrt{u^2+A_1+A_2}$ by \eqref{eq:radii-slaved}.
The last integral is finite because $R_3(u)\le R_3(0)$ and
$u^2\log(R_3(0)/u)$ is bounded on $(0,R_1(0)]$.
Hence, the central disk $\interface{1}$ vanishes at the finite time $t_\star$.
\end{proof}

\begin{remark}
  The motion law of three circles \eqref{eq:base-flow} was shown in \cite[Appendix A]{EGN24},
  although therein, the direction of each normal vector $\normal{j}$ was taken differently, and there was no argument that showed the finite-time disappearance of the smallest circle $\interface{1}$.
\end{remark}

The two-phase thin annulus nevertheless exhibits a singular linearized mechanism: as shown in Section~\ref{sec:twophase},
its fixed-mode matrix has eigenvalues of order $\delta^{-1}$ at every wave number $k\ge2$, including a growing branch with a positive rate relative to the mean thickness.
On the other hand, in the three-concentric-circle case, the mean thickness of the layer increases along the exact solution, as stated below.
\begin{corollary}\label{cor:hC}
The thickness $h:=R_2-R_1$ of the layer $\region{1}$ satisfies
\begin{equation}\label{eq:hC}
 \frac{\dd}{\dd t}\log h=\frac{\alpha}{R_1R_2}>0.
\end{equation}
\end{corollary}

\begin{proof}
By \eqref{eq:base-flow}, we have
$\dot h=\alpha(1/R_1-1/R_2)=\alpha h/(R_1R_2)$.
\end{proof}

\section{Fourier linearization and the mode matrix}\label{sec:linearization}

The concentric solution of Section~\ref{sec:setting} is radially
symmetric at every instant, and its layer $\region{1}$ only widens
(Corollary~\ref{cor:hC}); no such evolution can bring the two circles
bounding the layer closer to each other.  By itself it therefore says
nothing about the question of Figure~\ref{fig:question}(c), whether a
modulation of the thickness is amplified.  The first step towards an
answer is to describe the flow in a neighborhood of the exact solution,
that is, to linearize it.

Within a single mode $\cos k\theta$ the perturbed potentials are
harmonic functions with explicit radial profiles, and the linearized
interface conditions involve only their values and radial derivatives on
the three circles.  The Dirichlet-to-Neumann maps convert the harmonic
problem into a linear system for six scalar Fourier coefficients of the
one-sided reduced-potential boundary values
(Proposition~\ref{prop:sixtrace}); solving it gives the mode matrix
$M_k$ of the Introduction.  The last subsection passes to the
coordinates adapted to the layer and reads off the gap row of $M_k$,
whose rates are the ones compared in Theorem~\ref{thm:main}(c).

\subsection{Mode ansatz and linearized interface conditions}

We perturb the three circles radially.
To this end, let $\bvp a:[0,\infty)\ni t\mapsto (a_1(t),a_2(t),a_3(t))^\top\in \R^3$ and set
\begin{equation}\label{eq:ansatz}
 r_j(\theta,t) := R_j(t)+a_j(t)\cos k\theta,\qquad j=1,2,3,\quad k\ge1,
\end{equation}
and write the perturbed potentials as

\begin{equation}\label{eq:ansatz-2}
\chemical{p}(\bv x)=\chemical{p}^0(|\bv x|)+\hat w_p(|\bv x|,t)\cos k\theta+O(|\bvp a|^2),
\end{equation}
where $\chemical{p}^0$ is the concentric solution derived in Proposition~\ref{prop:base};
the second term denotes the first-order perturbation of the chemical potential induced by the interface amplitudes $\bvp a$.
The third term will be discarded in the formal amplitude expansion.

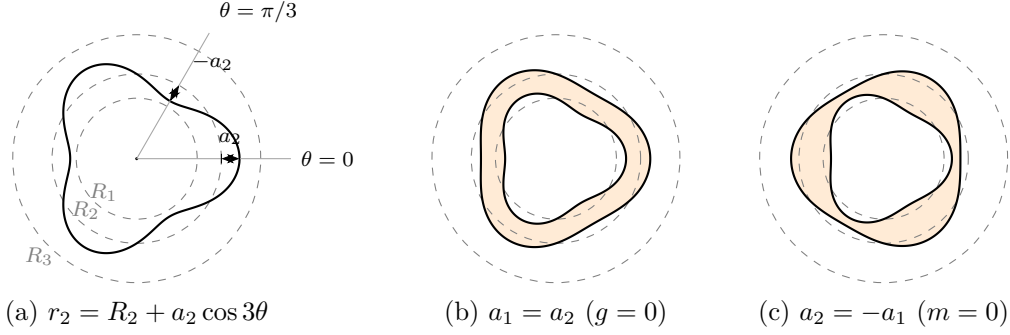
\begin{figure}[htb]
\centering
\begin{tikzpicture}[scale=0.80,>=Stealth]
 \def\Ra{1.00}\def\Rb{1.40}\def\Rc{2.05}\def\ab{0.30}
 \draw[dashed,gray] (0,0) circle (\Ra);
 \draw[dashed,gray] (0,0) circle (\Rb);
 \draw[dashed,gray] (0,0) circle (\Rc);
 \draw[thick] plot[domain=0:360,samples=181,smooth]
  ({(\Rb+\ab*cos(3*\x))*cos(\x)},{(\Rb+\ab*cos(3*\x))*sin(\x)});
 \fill (0,0) circle (0.025);
 \draw[thin,gray!70] (0,0) -- (2.55,0);
 \node[font=\scriptsize,right] at (2.55,0) {$\theta=0$};
 \draw[thin] (\Rb,-0.08) -- (\Rb,0.08);
 \draw[thin] (\Rb+\ab,-0.08) -- (\Rb+\ab,0.08);
 \draw[<->,thin] (\Rb,0) -- (\Rb+\ab,0);
 \node[font=\scriptsize,above=1.5pt] at (\Rb+0.5*\ab,0) {$a_2$};
 \coordinate (P) at ({\Rb*cos(60)},{\Rb*sin(60)});
 \coordinate (Q) at ({(\Rb-\ab)*cos(60)},{(\Rb-\ab)*sin(60)});
 \draw[thin,gray!70] (0,0) -- ({2.4*cos(60)},{2.4*sin(60)});
 \node[font=\scriptsize,anchor=south west] at ({2.4*cos(60)},{2.4*sin(60)}) {$\theta=\pi/3$};
 \draw[<->,thin] (Q) -- (P);
 \node[font=\scriptsize,anchor=south west] at ({1.50*cos(60)},{1.50*sin(60)}) {$-a_2$};
 \node[font=\scriptsize,gray] at ({0.78*cos(225)},{0.78*sin(225)}) {$R_1$};
 \node[font=\scriptsize,gray] at ({1.20*cos(225)},{1.20*sin(225)}) {$R_2$};
 \node[font=\scriptsize,gray] at ({2.28*cos(225)},{2.28*sin(225)}) {$R_3$};
 \node[font=\small] at (0,-2.55) {(a) $r_2=R_2+a_2\cos3\theta$};
\end{tikzpicture}%
\hspace{0.9cm}%
\begin{tikzpicture}[scale=0.80]
 \filldraw[even odd rule,fill=orange!16,draw=none]
  plot[domain=0:360,samples=181,smooth]
   ({(1.40+0.16*cos(3*\x))*cos(\x)},{(1.40+0.16*cos(3*\x))*sin(\x)}) -- cycle
  plot[domain=0:360,samples=181,smooth]
   ({(1.00+0.16*cos(3*\x))*cos(\x)},{(1.00+0.16*cos(3*\x))*sin(\x)}) -- cycle;
 \draw[dashed,gray] (0,0) circle (1.00);
 \draw[dashed,gray] (0,0) circle (1.40);
 \draw[dashed,gray] (0,0) circle (2.05);
 \draw[thick] plot[domain=0:360,samples=181,smooth]
  ({(1.00+0.16*cos(3*\x))*cos(\x)},{(1.00+0.16*cos(3*\x))*sin(\x)});
 \draw[thick] plot[domain=0:360,samples=181,smooth]
  ({(1.40+0.16*cos(3*\x))*cos(\x)},{(1.40+0.16*cos(3*\x))*sin(\x)});
 \node[font=\small] at (0,-2.55) {(b) $a_1=a_2$\ ($g=0$)};
\end{tikzpicture}%
\hspace{0.9cm}%
\begin{tikzpicture}[scale=0.80]
 \filldraw[even odd rule,fill=orange!16,draw=none]
  plot[domain=0:360,samples=181,smooth]
   ({(1.40-0.13*cos(3*\x))*cos(\x)},{(1.40-0.13*cos(3*\x))*sin(\x)}) -- cycle
  plot[domain=0:360,samples=181,smooth]
   ({(1.00+0.13*cos(3*\x))*cos(\x)},{(1.00+0.13*cos(3*\x))*sin(\x)}) -- cycle;
 \draw[dashed,gray] (0,0) circle (1.00);
 \draw[dashed,gray] (0,0) circle (1.40);
 \draw[dashed,gray] (0,0) circle (2.05);
 \draw[thick] plot[domain=0:360,samples=181,smooth]
  ({(1.00+0.13*cos(3*\x))*cos(\x)},{(1.00+0.13*cos(3*\x))*sin(\x)});
 \draw[thick] plot[domain=0:360,samples=181,smooth]
  ({(1.40-0.13*cos(3*\x))*cos(\x)},{(1.40-0.13*cos(3*\x))*sin(\x)});
 \node[font=\small] at (0,-2.55) {(c) $a_2=-a_1$\ ($m=0$)};
\end{tikzpicture}
\caption{Fourier perturbation of the concentric state ($k=3$); dashed:
reference circles. (a) The ansatz \eqref{eq:ansatz} for $\interface{2}$ alone:
$a_2$ is the signed radial displacement from $R_2$, outward at $\theta=0$ and
inward at $\theta=\pi/3$. (b),\,(c) $\interface{1}$ and $\interface{2}$
perturbed in the two coordinate directions of \eqref{eq:mg}.}
\label{fig:modes}
\end{figure}

\begin{lemma}[Mode separation and mode shapes]\label{lem:modes}
The linearization about the concentric solution closes exactly within
each Fourier mode $k\ge1$, and within the $\cos$ and $\sin$ families
separately; the $\sin$ family satisfies the same equations by rotation equivariance.
The radial profiles solve
$$\hat w_p''+\frac1r\hat w_p'-\frac{k^2}{r^2}\hat w_p=0,$$
and hence, it holds that 
\begin{equation}\label{eq:mode-shapes}
 \hat w_p=
 \begin{cases}
 \alpha_p r^k &\qquad\text{in the bounded component of }\region{3},\\
 \alpha_p r^k+\beta_p r^{-k} &\qquad\text{in}\quad\region{1}\cup\region{2},\\
 \beta_p r^{-k} &\qquad\text{in the unbounded component of }\region{3}.
 \end{cases}
\end{equation}
\end{lemma}

\begin{proof}
The base state and its linearized boundary coefficients are radial,
so Fourier orthogonality makes every sine and cosine mode invariant,
and a rotation carries one family to the other.

\pstep{1}{Radial equation}
Write $\cos k\theta=P_k(x,y)/r^k$ with $P_k:=\operatorname{Re}(x+iy)^k$,
a harmonic polynomial, homogeneous of degree $k$, so that
$x\partial_xP_k+y\partial_yP_k=kP_k$.  For a radial factor $u(r)$ one has
$\nabla u=u'(r)\,(x,y)/r$ and $\Delta u=u''+u'/r$, hence on $\R^2\setminus\{0\}$
\[
 \Delta(uP_k)=\Bigl(u''+\frac{u'}r\Bigr)P_k+\frac{2u'}{r}\,(x\partial_x+y\partial_y)P_k+u\,\Delta P_k
 =\Bigl(u''+\frac{(2k+1)u'}{r}\Bigr)P_k .
\]
With $u=\hat w_pr^{-k}$ this is
\[
 \Delta\bigl(\hat w_p(r)\cos k\theta\bigr)
 =\Bigl(\hat w_p''+\frac1r\hat w_p'-\frac{k^2}{r^2}\hat w_p\Bigr)\cos k\theta .
\]
The perturbed potential is harmonic in each region, so the coefficient
of $\cos k\theta$ vanishes: this is the displayed radial equation.

\pstep{2}{Mode shapes}
With $t=\log r$ and $v(t)=\hat w_p(e^t)$ the radial equation becomes
$v''=k^2v$.  Put $f:=v'+kv$; then $f'=k f$, so $f=Ce^{kt}$, and
$(e^{kt}v)'=e^{kt}f=Ce^{2kt}$ gives $e^{kt}v=\frac{C}{2k}e^{2kt}+D$, i.e.\
$\hat w_p=\frac C{2k}r^k+Dr^{-k}$.  Boundedness near the origin forces
$D=0$ in the disk, and boundedness at infinity, \eqref{eq:interfaced-FX},
forces the coefficient of $r^k$ to vanish in the exterior; this is
\eqref{eq:mode-shapes}.
\end{proof}

\begin{remark}[Exclusion of the mode $k=0$]\label{rem:kzero}
(i) The mode $k=0$ is not a shape perturbation.
Displacing the three circles by constants, that is $r_j=R_j+a_j$, produces again three concentric circles;
another member of the family of exact solutions of Proposition~\ref{prop:base} with different initial data:
two of the three directions $\bvp a\in\R^3$ change the conserved annular areas, and
the third, tangent to the orbit, moves the state along the same orbit,
equivalently shifts the extinction time $t_\star$.
The layer keeps a uniform thickness, and its evolution is given exactly by Proposition~\ref{prop:base}.
All statements about shape modes therefore start at $k=1$.
\end{remark}

In the sequel, let us derive the linearized counterpart of the interface conditions \eqref{eq:interfaces}
in terms of the amplitude $\bvp a$ and the radial chemical potential $w^0_p$.
Let $c_\theta := \cos k\theta$, $s_\theta := \sin k\theta$, $\rho(\theta):=R_j+a_jc_\theta$ and
$X(\theta):=\rho(\theta)(\cos\theta,\sin\theta)^\top$.
For fixed $k$ and $|a_j|<R_j$, we see that $\rho(\theta)>0$.
The direct differentiation gives
\begin{equation*}
 \begin{aligned}
 X'&=\rho'\bv e_r+\rho \bv e_\theta,
 &\rho'&=-a_jks_\theta,
 &\rho''&=-a_jk^2c_\theta,\\
 |X'|&=\sqrt{\rho^2+(\rho')^2},
 \end{aligned}
\end{equation*}
where $\bv e_r=(\cos\theta,\sin\theta)^\top$ and
$\bv e_\theta=(-\sin\theta,\cos\theta)^\top$.
Let $\phi(\theta)$ be the angle between the tangent vector $X'$ and the unit vector $\bv e_1 = (1,0)^\top$.
Since $\rho>0$, the signed angle from $\bv e_\theta$ to $X'$ is equal to $-\arctan(\rho'/\rho)$, and hence
\begin{equation*}
 \phi(\theta)=\theta+\frac{\pi}{2}-\arctan\frac{\rho'}{\rho},
 \qquad
 \frac{\dd\phi}{\dd\theta}
 =1-\frac{\rho\rho''-(\rho')^2}{\rho^2+(\rho')^2}
 =\frac{\rho^2+2(\rho')^2-\rho\rho''}{\rho^2+(\rho')^2}.
\end{equation*}
Let $s$ be ordinary arclength, so that $\dd s/\dd\theta=|X'|$.
With the outward-normal convention of Section~\ref{sec:setting},
we deduce that
\begin{equation*}
 \curvature{j}
 = -\frac{\dd \phi}{\dd s} = -\frac{\dd \phi}{\dd \theta}\cdot\frac{\dd \theta}{\dd s}
 =-\frac{\rho^2+2(\rho')^2-\rho\rho''}
         {\bigl(\rho^2+(\rho')^2\bigr)^{3/2}}.
\end{equation*}
We expand the numerator and denominator above in terms of the amplitude $a_j$:
\begin{equation*}
 \begin{aligned}
 \rho^2+2(\rho')^2-\rho\rho''
  &=R_j^2+R_ja_j(k^2+2)c_\theta+O(a_j^2),\\
 \bigl(\rho^2+(\rho')^2\bigr)^{3/2}
  &=\bigl(R_j^2+2R_ja_jc_\theta+O(a_j^2)\bigr)^{3/2}
  =R_j^3\left(1+\frac{2a_j}{R_j}c_\theta+O(a_j^2)\right)^{3/2}\\
  &=R_j^3\left(1+\frac32\frac{2a_j}{R_j}c_\theta\right)+O(a_j^2)
  =R_j^3\left(1+\frac{3a_j}{R_j}c_\theta\right)+O(a_j^2).
 \end{aligned}
\end{equation*}
Using $(1+z)^{-1}=1-z+O(z^2)$ yields
\begin{equation}
  \label{eq:curvature-exp}
 \begin{aligned}
 \curvature{j}
 &= -\frac{1}{R_j}\left(\frac{1 + \frac{a_j(k^2 + 2)}{R_j}c_\theta + O(a_j^2)}{1 + \frac{3a_j}{R_j}c_\theta + O(a_j^2)}\right) \\
 &=-\frac1{R_j}
   \left(1+\frac{a_j(k^2+2)}{R_j}c_\theta\right)
   \left(1-\frac{3a_j}{R_j}c_\theta\right)+O(a_j^2)
 =-\frac1{R_j}-\frac{k^2-1}{R_j^2}a_jc_\theta+O(a_j^2).
 \end{aligned}
\end{equation}

The outward unit normal and normal velocity of the curve $\interface{j}$ are
\begin{equation}
  \label{eq:anzatz-3}
 \begin{aligned}
 \normal{j}
 &= \rotation{-\tfrac{\pi}{2}}\left(\frac{\derivTheta{X}}{|\derivTheta{X}|}\right)
 = \frac{\derivTheta{\rho_j}\rotation{-\tfrac{\pi}{2}} \bv e_r + \rho_j\rotation{-\tfrac{\pi}{2}}\bv e_\theta}{\sqrt{\rho^2_j + (\derivTheta{\rho_j})^2}}
 =\frac{\rho_j\bv e_r-(\derivTheta{\rho_j})\bv e_\theta}
         {\sqrt{\rho_j^2+(\derivTheta{\rho_j})^2}}
  =\bv e_r+O(|\bvp a|),\\
 V_j
 &= \derivTime{X}\cdot\normal{j} = \derivTime{\rho_j}\,\bv e_r\cdot\normal{j}
  =\derivTime{R_j}+\derivTime{a_j}c_\theta+O(|\bvp a|^2),
 \end{aligned}
\end{equation}
where $\rotation{\varphi}\in\R^{2\times 2}$ denotes the rotational matrix with the angle $\varphi$ counter-clockwise.

For a smooth radial function $F$, Taylor's theorem on the displaced interface gives
\begin{equation}\label{eq:dkappa}
 F(R_j+a_jc_\theta) = F(R_j)+a_j\,\partial_rF(R_j)c_\theta+O(a_j^2).
\end{equation}
Applying \eqref{eq:dkappa} with $F=\chemical{p}^{0,\pm}$ and
$F=\hat w_p^\pm$ gives the value on the displaced interface:
\begin{equation}
  \label{eq:linearized-w}
  \begin{aligned}
 \chemical{p}^\pm(R_j+a_jc_\theta,\theta)
 &= \chemical{p}^{0,\pm}(R_j + a_jc_\theta) + \hat w_p^\pm(R_j + a_jc_\theta)c_\theta+ O(|\bvp a|^2)\\
 &=\chemical{p}^{0,\pm}(R_j)
   +a_j\partial_r\chemical{p}^{0,\pm}(R_j)c_\theta
   +\left(\hat w_p^\pm(R_j)
     +a_j\partial_r\hat w_p^\pm(R_j)c_\theta\right)c_\theta
   +O(|\bvp a|^2)\\
 &=\chemical{p}^{0,\pm}(R_j)
  +\left(\hat w_p^\pm(R_j)
    +a_j\partial_r\chemical{p}^{0,\pm}(R_j)\right)c_\theta
  +O(|\bvp a|^2).
  \end{aligned}
\end{equation}
Here, we have used $\hat w_p = O(|\bvp a|)$ (recall \eqref{eq:ansatz-2}) to obtain $a_j\partial_r\hat w_p^\pm(R_j)c_\theta^2 = O(|\bvp a|^2)$.

Meanwhile, we compute
\begin{equation*}
 \nabla\chemical{p}^\pm(r,\theta)
 =\left(\partial_r\chemical{p}^{0,\pm}(r)
       +\partial_r\hat w_p^\pm(r)c_\theta\right)\bv e_r
   +\frac{\hat w_p^\pm(r)}{r}\derivTheta{c_\theta}\,\bv e_\theta
   +O(|\bvp a|^2).
\end{equation*}
Since $\hat w_p^\pm=O(|\bvp a|)$, we deduce from the formula \eqref{eq:anzatz-3} for $\normal{j}$ that
\begin{align}
  \label{eq:linearized-gw}
 \nabla\chemical{p}^\pm(\rho_j,\theta)\cdot\normal{j}
 &=\partial_r\chemical{p}^{0,\pm}(\rho_j)
   +\partial_r\hat w_p^\pm(\rho_j)c_\theta
   +O(|\bvp a|^2)\nonumber\\
 &=\partial_r\chemical{p}^{0,\pm}(R_j)
   +a_j\partial_r^2\chemical{p}^{0,\pm}(R_j)c_\theta
   +\bigl[\partial_r\hat w_p^\pm(R_j)
      +a_j\partial_r^2\hat w_p^\pm(R_j)c_\theta\bigr]c_\theta
   +O(|\bvp a|^2)\nonumber\\
 &=\partial_r\chemical{p}^{0,\pm}(R_j)
   +\bigl[\partial_r\hat w_p^\pm(R_j)
      +a_j\partial_r^2\chemical{p}^{0,\pm}(R_j)\bigr]c_\theta
   +O(|\bvp a|^2).
\end{align}
Here, we have invoked \eqref{eq:dkappa} with $F=\partial_r\chemical{p}^{0,\pm}$ and $F=\partial_r\hat w_p^\pm$ together with
$a_j\partial_r^2\hat w_p^\pm = O(|\bvp a|^2)$ to obtain the second equality.

We now take the jump of the value expansion \eqref{eq:linearized-w}.
Since $\jump{\chemical{p}^0}=0$,
for the continuity condition \eqref{eq:interfaces-C},
we require

\begin{equation}
  \label{eq:linearized-C-pre}
\jump{\hat w_p}+a_j\jump{\partial_r\chemical{p}^0}=0.
\end{equation}

For the Gibbs--Thomson condition \eqref{eq:interfaces-GT}, write
$\hat{\bvp w}=(\hat w_1,\hat w_2,\hat w_3)^\top$.
Let $p_-:=\indexphasem{j}$ and $p_+:=\indexphasep{j}$.
Since $\jump{\chi_{p_-}}=-1$ and $\jump{\chi_{p_+}}=1$,
\eqref{eq:interfaces-GT} turns to
$w_{p_-}-w_{p_+}=-\tension{j}\curvature{j}$.
Subtracting the concentric identity
$(w_{p_-}^0-w_{p_+}^0)(R_j)=\tension{j}/R_j$ from
the scalar form of \eqref{eq:interfaces-GT} together with \eqref{eq:linearized-w},
using the curvature expansion \eqref{eq:curvature-exp}, and comparing the coefficients of
$c_\theta$, we obtain
\begin{equation}
  \label{eq:linearized-GT-pre}
 \left.\left(\hat w_{p_-}-\hat w_{p_+}
   +a_j\partial_r(w_{p_-}^0-w_{p_+}^0)\right)\right|_{r=R_j}
 =\frac{\tension{j}(k^2-1)}{R_j^2}a_j.
\end{equation}
Here, we note that the continuity relation \eqref{eq:linearized-C-pre} makes
$\hat w_p+a_j\partial_r w_p^0$ have the same trace from both sides.

For $q=q(j)$, we recall the no-flux identity \eqref{eq:ML-components}, that is, $\jump{\nabla\chemical{q}}\cdot\normal{j}=0$.
We take the jump of the normal derivative expansion \eqref{eq:linearized-gw},
use $\normal{j}=\bv e_r$ on the concentric circles
(by \eqref{eq:anzatz-3}) to obtain
$\jump{\partial_r\chemical{q}^0}=0$ from
\eqref{eq:ML-components}, then subtract this zeroth-order identity
and compare the coefficients of $c_\theta$.
This gives

\begin{equation}
  \label{eq:linearized-FX-pre}
  \jump{\partial_r\hat w_q}+a_j\jump{\partial_r^2\chemical{q}^0}=0.
\end{equation}

For $p_-:=\indexphasem{j}$, the same motion law and
$\jump{\chi_{p_-}}=-1$ give
$V_j=\jump{\nabla\chemical{p_-}}\cdot\normal{j}$.
We use the normal velocity in \eqref{eq:anzatz-3} and the normal
derivative expansion \eqref{eq:linearized-gw},
subtract the base relation $\dot R_j=\jump{\partial_r\chemical{p_-}^0}$, and compare the
coefficients of $c_\theta$ to obtain

\begin{equation}
  \label{eq:linearized-ML-pre}
 \derivTime{a}_j=\jump{\partial_r\hat w_{p_-}}
   +a_j\,\jump{\partial_r^2\chemical{p_-}^0}.
\end{equation}

We summarize the linearized interface condition on the interface $\interface{j}$ as follows:
\begin{subequations}\label{eq:linear-a}
\begin{align}
 \jump{\hat w_p}&=-a_j\,\jump{\partial_r \chemical{p}^0},
 \label{eq:pc}\\
 \bigl(\hat{\bvp w}
       +a_j\partial_r\bvp w^0\bigr)
       \cdot\jump{\bvp\chi}
 &=-\frac{\tension{j}(k^2-1)}{R_j^2}a_j,
 \label{eq:pgt}\\
 \jump{\partial_r\hat w_q}&=-a_j\,\jump{\partial_r^2\chemical{q}^0},
 \label{eq:pfx}\\
 \derivTime{a}_j&=\jump{\partial_r\hat w_{p_-}}
   +a_j\,\jump{\partial_r^2\chemical{p_-}^0}.
 \label{eq:pv}
\end{align}
\end{subequations}

\subsection{Annular Dirichlet-to-Neumann maps and the six-coefficient reduction}
The linearized conditions \eqref{eq:linear-a} involve the perturbed potentials $\hat w_p$ only through their values and radial derivatives on the three circles.
In this subsection, we eliminate the derivatives by means of the Dirichlet-to-Neumann maps of the four radial regions and
reduce the linearized problem to a finite system for six scalar
Fourier coefficients of the one-sided reduced-potential boundary values
at the three circles.  Solving that system gives the mode matrix $M_k$.

\begin{lemma}[Annular Dirichlet-to-Neumann map]\label{lem:dtn}
Let $k\ge1$ be an integer, and let $0<R_a<R_b$ be the inner and outer radii of an annulus.
Let $\Phi(r,\theta)=U(r)\cos(k\theta)$ be a harmonic Fourier mode on the annulus,
where the scalar function $U(r)=A r^k+B r^{-k}$ represents the radial amplitude for some constants $A,B$.
Set $S_a:=U(R_a)$, $S_b:=U(R_b)$, and $L:=\log(R_b/R_a)$.

Then, every pair $(S_a,S_b)^\top\in\mathbb R^2$ determines a unique harmonic Fourier mode of the stated form on the annulus.
For a disk $\ball{R_{in}}$ or an exterior region $\R^2\setminus\ball{R_{out}}$ for some $R_{in},\,R_{out}>0$,
let $U$ denote the radial amplitude of a mode regular at the origin or decaying at infinity,
respectively, and set $S:=U(R_{in})$ in the disk and
$S:=U(R_{out})$ in the exterior.
This single boundary coefficient also determines the corresponding mode uniquely.
The coefficients of the radial derivatives, in the direction of increasing $r$, are
\begin{equation}\label{eq:dtn}
  U'(r) =
\begin{cases}
 \dfrac{k}{R_{in}}S,
   &\text{if }r=R_{in},\\
 \dfrac{k}{R_a \sinh(kL)}(S_b-S_a\cosh(kL)),
   &\text{if }r=R_a,\\
 \dfrac{k}{R_b \sinh(kL)}(S_b\cosh(kL)-S_a),
   &\text{if }r=R_b,\\
 -\dfrac{k}{R_{out}}S,
   &\text{if }r=R_{out}.
\end{cases}
\end{equation}
In particular, the outward-normal derivative coefficients of the annulus are
$-U'(R_a)$ and $U'(R_b)$, respectively.
\end{lemma}

\begin{proof}
By Lemma~\ref{lem:modes}, the annular radial amplitude satisfies
$U'' + r^{-1}U' - k^2r^{-2}U = 0$.
Put $t=\log(r/R_a)$ and $v(t)=U(R_a e^t)$.
We compute $v'(t) = R_a e^t U'(R_a e^t)$, and hence
\begin{equation*}
 v''(t) = R_a e^t U'(R_a e^t) + (R_a e^t)^2 U''(R_a e^t) = r U'(r) + r^2 U''(r) = k^2U(r) = k^2 v(t)
\end{equation*}
on noting that $r = R_a e^t$.
With the boundary conditions $v(0)=S_a$ and $v(L)=S_b$, we obtain
\begin{equation}\label{eq:formula-v}
 v(t)=\frac{\sinh(k(L-t))S_a +\sinh(kt)S_b}{\sinh(kL)}.
\end{equation}
This formula gives the unique solution to $v'' = k^2 v$ for every pair $(S_a,S_b)$;
with the prescribed boundary values, proving the first claim.
The direct differentiation of \eqref{eq:formula-v} and the identity $U'(r)=v'(t)/r$ evaluated at $t=0$ and $t=L$
give the two annular rows of \eqref{eq:dtn} for the cases of $r = R_a$ and $r = R_b$.
Regularity at the origin and decay at infinity leave the radial amplitudes
$U(r)=C_{\rm in}r^k$ in the disk and $U(r)=C_{\rm out}r^{-k}$ in the exterior.
Using their boundary coefficients $S$ gives
\begin{align*}
 U'(R_{in})&=kC_{\rm in}R_{in}^{k-1}
   =\frac{k}{R_{in}}S,
   && S=U(R_{in}) = C_{\rm in}R_{in}^{k},\\
 U'(R_{out})&=-kC_{\rm out}R_{out}^{-k-1}
   =-\frac{k}{R_{out}}S,
   && S=U(R_{out}) =C_{\rm out}R_{out}^{-k}.
\end{align*}
The proof is now complete.
\end{proof}

We now apply the mode-wise Dirichlet-to-Neumann formulas to the linearized interface
conditions \eqref{eq:linear-a}.  For a prescribed circle-amplitude
vector $a$, it is enough to determine the Fourier coefficients of the
reduced potential at the three circles: the Dirichlet-to-Neumann maps then give all
radial derivatives, and the velocity conditions give $\derivTime{a}$.
The next definition names these boundary coefficients and separates
them from the known concentric-state data.

\begin{definition}[Boundary Fourier coefficients and base data]\label{def:traces}
Fix one Fourier mode $k\ge1$ and set
$\hat{\bvp W}(r):=(\hat w_1(r),\hat w_2(r))^\top\in\R^2$.
This is the radial coefficient vector of the reduced potential perturbation;
the angular factor is $\cos(k\theta)$.  At each reference circle
$\interface{j}$, define the coefficient vectors on its inner and outer sides by
\begin{equation*}
 \bvp u_j^\pm:=\hat{\bvp W}(R_j)\big|_{\pm},\qquad
 \bvp v_j^\pm:=\partial_r\hat{\bvp W}(R_j)\big|_{\pm},
 \qquad j=1,\,2,\,3.
\end{equation*}
Set $\bvp u^-:=(\bvp u_1^-,\bvp u_2^-,\bvp u_3^-)\in\R^6$.
In terms of the derivative jumps $J_j$ of \eqref{eq:base-jumps}, we define
\begin{equation}\label{eq:JHG}
 H_j := \jump{\partial_r^2\bvp W^0}_{\interface{j}}=-\frac{J_j}{R_j},\qquad
 G_j := \partial_r\bvp W^0(R_j)\big|_{\rm in},\qquad
 \qquad j=1,\,2,\,3,
\end{equation}
so that $$G_1=0,\quad G_2=\frac{\alpha}{R_2}c_1,\quad G_3=\frac{\alpha}{R_3}c_2.$$
\end{definition}

For fixed reference radii and tensions, $J_j$, $G_j$,
and $H_j$ are supposed to be given vectors in $\R^2$, and
the unknown boundary coefficients are the six components of $\bvp u^-$.

The continuity condition \eqref{eq:pc} first eliminates the outer-side
boundary coefficients:
$\bvp u_j^+=\bvp u_j^--a_jJ_j$ for $j=1,2,3$.
The two annular Dirichlet-to-Neumann maps in Lemma~\ref{lem:dtn}, together with its disk
and exterior formulas, then eliminate all radial-derivative coefficients:
\begin{equation*}
\begin{array}{c|c|c}
 \text{radial region}&\text{boundary coefficients}&\text{radial-derivative coefficients}\\ \hline
 0<r<R_1&\bvp u_1^-&\bvp v_1^-\\
 R_1<r<R_2&(\bvp u_1^+,\bvp u_2^-)&(\bvp v_1^+,\bvp v_2^-)\\
 R_2<r<R_3&(\bvp u_2^+,\bvp u_3^-)&(\bvp v_2^+,\bvp v_3^-)\\
 r>R_3&\bvp u_3^+&\bvp v_3^+
\end{array}
\end{equation*}
\begin{equation}\label{eq:vtraces}
 \bvp v_1^-=\frac{k}{R_1}\bvp u_1^-,\qquad
 \binom{\bvp v_1^+}{\bvp v_2^-}
   =\mathcal N^{(k)}_{R_1R_2}\binom{\bvp u_1^+}{\bvp u_2^-},\qquad
 \binom{\bvp v_2^+}{\bvp v_3^-}
   =\mathcal N^{(k)}_{R_2R_3}\binom{\bvp u_2^+}{\bvp u_3^-},\qquad
 \bvp v_3^+=-\frac{k}{R_3}\bvp u_3^+.
\end{equation}
Here, $\mathcal N^{(k)}_{R_aR_b}$ acts componentwise on the two reduced
chemical potentials.  With $L=\log(R_b/R_a)$, its $4\times4$ matrix is
\begin{equation*}
 \mathcal N^{(k)}_{R_aR_b}
 =\begin{pmatrix}
  -\dfrac{k\cosh(kL)}{R_a\sinh(kL)} & \dfrac{k}{R_a\sinh(kL)}\\[3pt]
  -\dfrac{k}{R_b\sinh(kL)} & \dfrac{k\cosh(kL)}{R_b\sinh(kL)}
 \end{pmatrix}\otimes I_2,
\end{equation*}
where $I_2$ denotes the identity matrix in $\R^{2\times 2}$.
The superscripts $-$ and $+$ in \eqref{eq:vtraces} name the two sides
of a circle; every $\bvp v_j^\pm$ is a derivative in the increasing
$r$ direction, including at an annulus's inner boundary.
The linearized Gibbs--Thomson and no-flux conditions supply two scalar
equations at each circle.  With the phase vectors of
\eqref{eq:phase-vectors}, they read
\begin{align}
 \bvp d_j^\top\bvp u_j^-&=\gamma_j a_j,
 \qquad \gamma_j:=\tension{j}\frac{k^2-1}{R_j^2}-\bvp d_j^\top G_j,
 \label{eq:gtrow}\\
 \bvp c_{q(j)}^\top(\bvp v_j^+-\bvp v_j^-)
 &=-\bvp c_{q(j)}^\top H_j\,a_j,
 \label{eq:fxrow}
\end{align}
for $j=1,2,3$.  Substituting $\bvp u_j^+=\bvp u_j^--a_jJ_j$
and \eqref{eq:vtraces} leaves precisely the six scalar unknowns in
$\bvp u^-$.  Stack the Gibbs--Thomson row and then the no-flux row
for each $j$ to define the $6\times6$ coefficient matrix
$\mathcal K_k(\bvp R)$ and the $6\times3$ shape-forcing matrix
$\mathcal P_k(R,\tensionVec)$:
\begin{equation}\label{eq:KB}
 \mathcal K_k(\bvp R)\bvp u^-
 =\mathcal P_k(R,\tensionVec)\bvp a.
\end{equation}
Equation \eqref{eq:KB} is the finite-dimensional static boundary system
for this Fourier mode.
The remaining three interface equations, the velocity laws \eqref{eq:pv},
then read
\begin{equation}\label{eq:DE}
 \dot a_j=\bvp c_{p_{\rm in}(j)}^\top
       (\bvp v_j^+-\bvp v_j^-)
       +\bvp c_{p_{\rm in}(j)}^\top H_j\,a_j,
 \qquad\text{i.e.}\qquad
 \dot{\bvp a}=\mathcal D_k\bvp u^-+\mathcal E_k\bvp a .
\end{equation}
Here $\mathcal D_k$ is $3\times6$ and $\mathcal E_k$ is $3\times3$;
they collect the coefficients after the same substitutions.

\begin{proposition}[Closed three-amplitude equation]\label{prop:sixtrace}
For every $k\ge1$ and $0<R_1<R_2<R_3$, the mode-$k$ linearized problem
uniquely determines the six boundary coefficients from $\bvp a$ and closes
on the three interface amplitudes:
\begin{equation}\label{eq:Mk}
 \dot{\bvp a}=M_k(\bvp R)\,\bvp a,\qquad
 M_k=\mathcal D_k\mathcal K_k^{-1}\mathcal P_k+\mathcal E_k ,
\end{equation}
where $\bvp R=(R_1,R_2,R_3)=\bvp R(t)$ along the concentric orbit.  The inverse exists by
Lemma~\ref{lem:global-invertibility}.
\end{proposition}

\begin{proof}
For each phase $p$, the Fourier coefficient of its chemical-potential
perturbation is $\hat w_p=\bvp c_p^\top\hat{\bvp W}$.
Hence, \eqref{eq:pc} gives
$\bvp u_j^+-\bvp u_j^-=-a_jJ_j$.
The disk, annular, and exterior formulas of Lemma~\ref{lem:dtn}
applied to these boundary values give \eqref{eq:vtraces}.
By \eqref{eq:phase-vectors}, the scalar product on the left of
\eqref{eq:pgt} equals
$-\bvp d_j^\top(\bvp u_j^-+a_jG_j)$.  Thus \eqref{eq:pgt} becomes
$\bvp d_j^\top(\bvp u_j^-+a_jG_j)
=\tension{j}(k^2-1)a_j/R_j^2$, which is \eqref{eq:gtrow}.
The no-flux condition \eqref{eq:pfx} gives \eqref{eq:fxrow} because
$\hat w_{q(j)}=\bvp c_{q(j)}^\top\hat{\bvp W}$.
For a concrete row of the finite system, put
$s_1=\sinh(k\log(R_2/R_1))$ and
$c_1=\cosh(k\log(R_2/R_1))$.  Eliminating $\bvp u_1^+$ in the
first no-flux row gives
\begin{equation*}
 \bvp c_{q(1)}^\top\left[
   \frac{k}{R_1s_1}\bvp u_2^-
   -\frac{k}{R_1}\left(1+\frac{c_1}{s_1}\right)\bvp u_1^-
 \right]
 =-\bvp c_{q(1)}^\top\left(H_1+\frac{kc_1}{R_1s_1}J_1\right)a_1.
\end{equation*}
The other no-flux rows are obtained in the same way.
All $\bvp v_j^\pm$ in these rows are linear in $(\bvp u^-,\bvp a)$
by \eqref{eq:vtraces}; collecting their $\bvp u^-$ and $\bvp a$
coefficients produces \eqref{eq:KB}.
The same substitution in the velocity laws yields \eqref{eq:DE}.
By Lemma~\ref{lem:global-invertibility}, $\mathcal K_k$ is invertible.
Thus $\bvp u^-=\mathcal K_k^{-1}\mathcal P_k\bvp a$ is the unique set of
boundary coefficients, and substitution into \eqref{eq:DE} gives
\eqref{eq:Mk}.
\end{proof}

The invertibility used in Proposition~\ref{prop:sixtrace} follows from
an energy identity for the homogeneous boundary system.

\begin{lemma}[Invertibility of the six-coefficient system]
\label{lem:global-invertibility}
For every integer $k\ge1$ and all $0<R_1<R_2<R_3$, the matrix
$\mathcal K_k(\bvp R)$ in \eqref{eq:KB} is invertible.  Hence, the mode
matrix $M_k(\bvp R)$ is well-defined at every point of the concentric orbit
before extinction.
\end{lemma}

\begin{proof}
It suffices to show that $\mathcal K_k\bvp u^-=0$ forces $\bvp u^-=0$,
so let $\bvp a=0$; then $\bvp u_j^+=\bvp u_j^-$ for every $j$.
For $p\in\{1,2,3\}$ and $j\in\{1,2,3\}$ write
$x_{p,j}:=\bvp c_p^\top\bvp u_j^-$ for the boundary coefficient of the
field $\hat w_p$ on $\interface{j}$ and
$J_{p,j}:=\bvp c_p^\top(\bvp v_j^+-\bvp v_j^-)$ for the jump of its
radial derivative, with $\bvp v_j^\pm$ given by \eqref{eq:vtraces}.
Since $c_1+c_2+c_3=0$, we have $\sum_px_{p,j}=0$ and $\sum_pJ_{p,j}=0$.

\pstep{1}{Quadratic forms}
For $L>0$ put
\[
 Q_L(S_a,S_b):=\frac{k}{\sinh kL}\Bigl(\cosh kL\,(S_a^2+S_b^2)-2S_aS_b\Bigr)
 =\frac{k}{\sinh kL}\Bigl((\cosh kL-1)(S_a^2+S_b^2)+(S_a-S_b)^2\Bigr).
\]
Since $\cosh kL>1$, $Q_L\ge0$ with equality only if $S_a=S_b=0$.
With $L_1=\log(R_2/R_1)$, $L_2=\log(R_3/R_2)$, define for each field
\[
 E_p:=k\,x_{p,1}^2+Q_{L_1}(x_{p,1},x_{p,2})+Q_{L_2}(x_{p,2},x_{p,3})+k\,x_{p,3}^2\ \ge 0 ,
\]
which vanishes only if $x_{p,1}=x_{p,2}=x_{p,3}=0$.

\pstep{2}{Energy identity}
Substituting the derivative coefficients \eqref{eq:vtraces}, i.e.\ the
formulas \eqref{eq:dtn} on the disk, the two annuli and the exterior,
into $J_{p,j}$, one verifies by direct expansion the identity
\begin{equation}\label{eq:energy-identity}
 E_p=-\sum_{j=1}^3R_j\,x_{p,j}J_{p,j},
 \qquad\text{hence}\qquad
 \sum_pE_p=-\sum_{j=1}^3R_j\,I_j,\quad
 I_j:=\sum_px_{p,j}J_{p,j}.
\end{equation}
(Each term of $E_p$ is the boundary form $\pm R\,U(R)U'(R)$ of the mode
$U$ of \eqref{eq:dtn} in one region, so $E_p$ is the Dirichlet energy of
$\Phi_p=\hat w_p\cos k\theta$ up to the factor $\pi$; this interpretation
is not needed below.)

\pstep{3}{Interface sums}
Fix $j$ and let $p_\pm=p_\pm(j)$, $q=q(j)$.  The Gibbs--Thomson row
\eqref{eq:gtrow} with $a_j=0$ reads $x_{p_-,j}=x_{p_+,j}$, and the no-flux
row \eqref{eq:fxrow} reads $J_{q,j}=0$.  With $\sum_pJ_{p,j}=0$ this gives
\begin{equation}\label{eq:homogeneous-energy-flux}
 I_j=x_{p_-,j}\bigl(J_{p_-,j}+J_{p_+,j}\bigr)+x_{q,j}J_{q,j}=0 .
\end{equation}

By \eqref{eq:energy-identity} and \eqref{eq:homogeneous-energy-flux},
$\sum_pE_p=0$.  Each $E_p$ is nonnegative, so every $E_p$ vanishes, hence
$x_{p,j}=0$ for all $p$ and $j$.  In particular
$\bvp u_j^-=(x_{1,j},x_{2,j})^\top=0$ for $j=1,2,3$, so $\bvp u^-=0$.
\end{proof}

The lemma gives pointwise invertibility for every strictly separated
geometry; it gives no bound on $\|\mathcal K_k^{-1}\|$ that is uniform
as $R_1\downarrow0$, as two circles coalesce, or as $k\to\infty$.

\begin{lemma}[Scaling and rationality]\label{lem:scaling}
$M_k$ is linear in the tension vector
$\tensionVec=(\tension{1},\tension{2},\tension{3})$ and has
length dimension $-3$:
\begin{equation}\label{eq:scaling}
 M_k(\bvp R,\tensionVec)=
 \frac{1}{R_1^3}\,\widehat M_k(\rho,\tau;\tensionVec),
 \qquad \rho=\frac{R_2}{R_1},\ \tau=\frac{R_3}{R_1}.
\end{equation}
For integer $k$ one has
$\sinh(k\log(R_b/R_a))=\tfrac12\{(R_b/R_a)^k-(R_a/R_b)^k\}$, so with
$e_1:=\rho^k$ and $e_2:=(\tau/\rho)^k$ as independent symbols the
entries of $\widehat M_k$ are rational in
$(k,\rho,\tau,e_1,e_2,\log\tau)$ and linear in $\tensionVec$.
\end{lemma}

\begin{proof}
Under $R_j\mapsto \ell R_j$, the base coefficient \eqref{eq:base-flow}
scales as $\ell^{-1}$ and is linear in $\tensionVec$; thus
the boundary Fourier coefficients per unit displacement
scale as $\ell^{-2}$ and
are linear in $\tensionVec$, while the velocity rows add one
inverse length.  Equations \eqref{eq:dtn}--\eqref{eq:Mk} therefore
give the factor $\ell^{-3}$ in \eqref{eq:scaling}.  Replacing each
hyperbolic factor by its two exponentials makes the
finite boundary-coefficient
system rational in the stated variables, and inversion preserves
rationality wherever the matrix is invertible, which
Lemma~\ref{lem:global-invertibility} guarantees.
\end{proof}

\subsection{Gap coordinates and instantaneous rates}

\begin{definition}[Gap row and instantaneous rates]\label{def:gap}
In the coordinates $(m,g,a_3)$ of \eqref{eq:mg}, that is
$a_1=m-g/2$ and $a_2=m+g/2$, write the gap row of $M_k$ as
\begin{equation}\label{eq:gaprow}
 (M_k\bvp a)_2-(M_k\bvp a)_1=\beta_m\,m+\beta_g\,g+\beta_3\,a_3 .
\end{equation}
The instantaneous pure-gap rate and the relative gap rate are
\begin{equation}\label{eq:rates}
 \lambda_k(\bvp R):=\beta_g,\qquad
 \Lambda_k(\bvp R):=\lambda_k(\bvp R)-\frac{\alpha}{R_1R_2}.
\end{equation}
\end{definition}

For a perturbation with $m=a_3=0$ one has $\dot g=\lambda_kg$ at the
given instant, so $\lambda_k$ is the logarithmic rate of a pure
thickness modulation, and $\Lambda_k$ compares it with the growth rate
\eqref{eq:hC} of the mean thickness.  Both identities are
instantaneous: along the orbit a pure-gap vector does not remain of
that form, and a general perturbation obeys the full row
\eqref{eq:gaprow}.

\begin{proposition}[Translation mode]\label{prop:translation}
It holds that $M_1\,(1,1,1)^\top=0$.
\end{proposition}

\begin{proof}
A rigid translation of the concentric solution by $\varepsilon$ along the
first axis shifts every circle by $\varepsilon\cos\theta$, so $\bvp a=\varepsilon(1,1,1)^\top$
with $k=1$, and perturbs the potential by
$-\varepsilon\,\partial_x\bvp w^0=-\varepsilon\,(\partial_r\bvp w^0)\cos\theta$,
whose boundary Fourier coefficients are $\bvp u_j^-=-G_j$, i.e.
\[
 \bvp u_1^-=0,\qquad \bvp u_2^-=-\frac{\alpha}{R_2}c_1,\qquad
 \bvp u_3^-=-\frac{\alpha}{R_3}c_2 .
\]
Direct substitution of these six coefficients and $\bvp a=(1,1,1)^\top$ into
\eqref{eq:KB} and \eqref{eq:DE} with $k=1$, where
$\sinh L=\frac12(R_b/R_a-R_a/R_b)$ and $\cosh L=\frac12(R_b/R_a+R_a/R_b)$
for $L=\log(R_b/R_a)$, verifies the six static equations and shows that
the three velocity rows vanish.  Since $\mathcal K_1$ is invertible
(Lemma~\ref{lem:global-invertibility}), $\bvp u^-=-G$ is the solution
of \eqref{eq:KB} for $\bvp a=(1,1,1)^\top$, and \eqref{eq:DE} gives
$M_1(1,1,1)^\top=0$.
\end{proof}

\section{A concentric two-phase comparison}\label{sec:twophase}

In this section phase $1$ occupies the annulus
$\ball{R_2}\setminus\overline{\ball{R_1}}$ and phase $3$ occupies
$\ball{R_1}\cup(\R^2\setminus\overline{\ball{R_2}})$; the inner circle
$\interface{1}$ carries the tension $\tension{\rm in}>0$ and the outer
circle $\interface{2}$ the tension $\tension{\rm out}>0$.  With two
phases the zero-sum representation reduces to one scalar potential
$w:=\chemical{3}=-\chemical{1}$, harmonic in each region and continuous
across both circles.  On $\interface{1}$ the inner phase is $3$ and
\eqref{eq:interfaces-GT} reads $2w=\tension{\rm in}/R_1$;
on $\interface{2}$ the inner phase is $1$ and \eqref{eq:interfaces-GT}
reads $-2w=\tension{\rm out}/R_2$. The concentric base state
is $w=a+b\log r$ in the annulus and constant in the disk and in the
exterior, with
\begin{equation}\label{eq:base2}
 b=-\bar\alpha,\qquad
 \bar\alpha=\frac{\tension{\rm in}/R_1+\tension{\rm out}/R_2}{2\log(R_2/R_1)},
 \qquad
 \dot R_j=-\frac{\bar\alpha}{R_j}\quad(j=1,2).
\end{equation}
The opposite Gibbs--Thomson values determine $b$ in \eqref{eq:base2}, and the velocity law gives $\dot R_j$; for equal tensions $\bar\alpha\sim\tensionsym/\delta$.

\begin{proposition}[Exact two-phase Fourier matrix]\label{prop:W5}
Let $g_1:=\dfrac{\tension{\rm in}(k^2-1)}{2R_1^2}$,
$g_2:=\dfrac{\tension{\rm out}(k^2-1)}{2R_2^2}$, $L=\log(R_2/R_1)$,
$s=\sinh kL$, and $c=\cosh kL$.  The mode matrix of the two-phase
annulus is
\begin{equation}\label{eq:W5}
 M^{(2)}_k=
 \begin{pmatrix}
 -\dfrac{kc}{R_1s}\Bigl(g_1+\dfrac{\bar\alpha}{R_1}\Bigr)-\dfrac{kg_1}{R_1}
   +\dfrac{\bar\alpha}{R_1^2}
 & -\dfrac{k}{R_1s}\Bigl(g_2-\dfrac{\bar\alpha}{R_2}\Bigr)\\[3mm]
 -\dfrac{k}{R_2s}\Bigl(g_1+\dfrac{\bar\alpha}{R_1}\Bigr)
 & -\dfrac{kc}{R_2s}\Bigl(g_2-\dfrac{\bar\alpha}{R_2}\Bigr)-\dfrac{kg_2}{R_2}
   +\dfrac{\bar\alpha}{R_2^2}
 \end{pmatrix}.
\end{equation}
\end{proposition}

\begin{proof}
The base annular derivative is $-\bar\alpha/r$.  At $R_1$,
\eqref{eq:pgt} on the disk side gives the boundary Fourier
coefficient $g_1a_1$ of $\hat w$, and \eqref{eq:pc} gives the annular
boundary coefficient
$S_1=(g_1+\bar\alpha/R_1)a_1$.  At $R_2$, \eqref{eq:pgt} on the
annular side gives $S_2=-(g_2-\bar\alpha/R_2)a_2$, and \eqref{eq:pc}
gives the exterior boundary coefficient $-g_2a_2$.
The annular Dirichlet-to-Neumann formula
\eqref{eq:dtn} therefore yields
\[
 v_1^+=\frac{k}{R_1s}(S_2-cS_1),\qquad
 v_2^-=\frac{k}{R_2s}(cS_2-S_1),
\]
while the disk and exterior derivatives are $kg_1a_1/R_1$ and
$kg_2a_2/R_2$.  At the inner circle the velocity law
\eqref{eq:pv} uses $\chemical{3}=w$, hence
$V_1=v_1^+-v_1^-+\bar\alpha a_1/R_1^2$; at the outer circle it uses
$\chemical{1}=-w$, hence $V_2=v_2^--v_2^++\bar\alpha a_2/R_2^2$.
Substituting $S_1,S_2$ gives each entry of \eqref{eq:W5}.
\end{proof}

The matrix \eqref{eq:W5} is explicit, and the thin-layer limit of
Theorem~\ref{thm:twophase} can be read off from it directly.  We prove
the theorem here, before the three-phase limit, because the
computation isolates what the two problems share and what separates
them.  With equal tensions the base coefficient \eqref{eq:base2} is
$\bar\alpha\sim\tensionsym/\delta$: the opposite Gibbs--Thomson values
on the two circles force a potential difference of order one across
the annulus, and the factor $\delta^{-1}$ of the annular
Dirichlet-to-Neumann map multiplies this difference.  The proof makes
the resulting order $\delta^{-1}$ of the rates explicit and shows at
the same time that the singular rates belong to the in-phase motion of
the two circles, not to the thickness.  In
Section~\ref{sec:thinlimit} the same factor $\delta^{-1}$ will
multiply a difference of order $\delta$.

\begin{proof}[Proof of Theorem~\ref{thm:twophase}]
By Lemma~\ref{lem:scaling} we may take $R=\tensionsym=1$ and write
$r=R_2=1+\varepsilon$; in general units every rate is multiplied by
$\tensionsym/R^3$ and $\varepsilon=\delta/R$.  Put $g=(k^2-1)/2$,
$s=\sinh(k\log r)$, $c=\cosh(k\log r)$, $\bar\alpha=(1+1/r)/(2\log r)$, and
\[
 A=g+\bar\alpha,\qquad B=\frac g{r^2}-\frac{\bar\alpha}r,\qquad
 H_1=-kg+\bar\alpha,\qquad H_2=-\frac{kg}{r^3}+\frac{\bar\alpha}{r^2},
\]
so that \eqref{eq:W5} reads
$M^{(2)}_k=\begin{pmatrix}-\frac{kc}{s}A+H_1&-\frac ksB\\-\frac{k}{rs}A&-\frac{kc}{rs}B+H_2\end{pmatrix}$.

\pstep{1}{Substitution}
Let
\[
 u:=\frac{\log(1+\varepsilon)}{\varepsilon}\to1,\qquad
 w_2:=\frac{(1+\varepsilon)^k-1-k\varepsilon}{\varepsilon^2}
 =\sum_{j=0}^{k-2}\binom{k}{j+2}\varepsilon^j\to\frac{k(k-1)}2
 \qquad\text{as}\quad\varepsilon\to 0,
\]
and $v:=k+\varepsilon w_2$, $e:=(1+\varepsilon)^k=1+\varepsilon v$,
$W:=v(2+\varepsilon v)=(e^2-1)/\varepsilon$.  Then
\[
 s=\frac{\varepsilon W}{2e},\qquad c=\frac{e^2+1}{2e},\qquad
 \varepsilon\bar\alpha=\frac{r+1}{2ru},
\]
and the singular factor $\varepsilon^{-1}$ of each entry of $M^{(2)}_k$
is now explicit.  Consequently, $\varepsilon\operatorname{tr}M^{(2)}_k$,
$\varepsilon^2\det M^{(2)}_k$, $\varepsilon^2M^{(2)}_k$ and
$\lambda_{\rm inst}$ are rational functions of $(\varepsilon,u,w_2)$
whose denominators ($u$, $W$, $r$, $e$) do not vanish at the limit
point $(0,1,k(k-1)/2)$, where $W=2k$ and $e=r=1$; for instance
\[
 \varepsilon\operatorname{tr}M^{(2)}_k
 =-\varepsilon kg\Bigl(1+\frac1{r^3}\Bigr)
  +\frac{r+1}{2ru}\Bigl(1+\frac1{r^2}\Bigr)
  -\frac{k(e^2+1)}{W}\left(g\Bigl(1+\frac1{r^3}\Bigr)+\frac{(r+1)^2}{2r^3u}\right),
\]
and the determinant, in which the cancellation $c^2-s^2=1$ removes the
apparent fourth-order pole, and $\lambda_{\rm inst}$ have the same form.
The limits are therefore the values at the limit point:
\begin{equation}\label{eq:twophase-limits}
  \begin{aligned}
 &\varepsilon\operatorname{tr}M^{(2)}_k\to-(k^2-1),\quad
 \varepsilon^2\det M^{(2)}_k\to-2(k^2-1),\\
 &\varepsilon^2M^{(2)}_k\to\begin{pmatrix}-1&1\\-1&1\end{pmatrix},\quad
 \lambda_{\rm inst}\to-\frac{k(k^2-1)}2\qquad\text{as}\quad \delta\downarrow 0.
  \end{aligned}
\end{equation}
The last limit is \eqref{eq:twophase-gap}.  Since each of these rational
functions is smooth at the limit point and $u-1$, $w_2-k(k-1)/2$ are
$O(\varepsilon)$, every limit in \eqref{eq:twophase-limits} holds with an
$O(\varepsilon)$ remainder.

\pstep{2}{Eigenvalues}
By \eqref{eq:twophase-limits},
$\varepsilon^2(\operatorname{tr}^2-4\det)\to(k^2-1)(k^2+7)>0$, so for all
small $\varepsilon$ the eigenvalues
$\lambda_\pm=\frac12(\operatorname{tr}\pm\sqrt{\operatorname{tr}^2-4\det})$
are real and distinct, and $\varepsilon\lambda_\pm\to\mu_\pm(k)$, the two
roots of $\mu^2+(k^2-1)\mu-2(k^2-1)$; this is \eqref{eq:thin2}.
Moreover, $\mu_+>1$ because
$(k^2-1)(k^2+7)-(k^2+1)^2=4(k^2-2)>0$ for $k\ge2$.
Since $\dot R_j=-\bar\alpha/R_j$, one has
$\frac{\dd}{\dd t}\log(R_2-R_1)=\bar\alpha/(R_1R_2)$, and
$\varepsilon\bar\alpha/r=(r+1)/(2r^2u)\to1$; hence
$\varepsilon\bigl(\lambda_+-\frac{\dd}{\dd t}\log(R_2-R_1)\bigr)\to\mu_+-1>0$,
which is \eqref{eq:twophase-relative}.

\pstep{3}{Eigenvectors}
Let $M^{(2)}_k\bvp v=\lambda\bvp v$ with $\lambda\in\{\lambda_+,\lambda_-\}$ and
$\bvp v\ne0$.  Multiplying the first row by $\varepsilon^2$ gives
\[
 v_2-v_1=-(\varepsilon^2M_{11}+1)\,v_1-(\varepsilon^2M_{12}-1)\,v_2
         +\varepsilon(\varepsilon\lambda)\,v_1,
 \qquad\text{hence}\qquad
 |v_2-v_1|\le\eta(\varepsilon)\max(|v_1|,|v_2|),
\]
with $\eta(\varepsilon):=|\varepsilon^2M_{11}+1|+|\varepsilon^2M_{12}-1|
+\varepsilon(|\varepsilon\lambda_+|+|\varepsilon\lambda_-|)=O(\varepsilon)$
by \eqref{eq:twophase-limits}.  If $v_1^2+v_2^2=1$ and
$\zeta=\operatorname{sgn}(v_1+v_2)$, then, using
$|v_1+v_2|^2=2-(v_2-v_1)^2$ and $|v_1+v_2|\le\sqrt2$,
\[
 \Bigl(v_1-\frac{\zeta}{\sqrt2}\Bigr)^2+\Bigl(v_2-\frac{\zeta}{\sqrt2}\Bigr)^2
 =2-\sqrt2\,|v_1+v_2|\le(v_2-v_1)^2\le\eta(\varepsilon)^2 .
\]
Thus, both normalized eigenvectors tend to $(1,1)/\sqrt2$ up to sign, and
the gap component $v_2-v_1$ of either mode is $O(\varepsilon)$.

Finally, for $k=1$ one has $g_1=g_2=0$ in \eqref{eq:W5}, and the
identities $c-s=R_1/R_2$ and $c+s=R_2/R_1$ give \eqref{eq:twophase-k1}
directly.  The square of this matrix is zero, and its pure-gap rate
vanishes identically.
\end{proof}

\begin{remark}[Real and distinct eigenvalues]\label{rem:twophase-real}
The proof gives slightly more than Theorem~\ref{thm:twophase}(a)
states.  By \eqref{eq:twophase-limits},
$\varepsilon^2(\operatorname{tr}^2-4\det)\to(k^2-1)(k^2+7)>0$ and
$\varepsilon^2\det M^{(2)}_k\to-2(k^2-1)<0$, so for all sufficiently
small $\delta$, depending on $k$, the two eigenvalues of $M^{(2)}_k$
are real, distinct and of opposite signs, $\lambda_-<0<\lambda_+$:
the thin two-phase annulus has exactly one growing and one decaying
mode at every $k\ge2$.  Moreover, since the limits in
\eqref{eq:twophase-limits} hold with $O(\varepsilon)$ remainders and
the discriminant is positive at the limit point, \eqref{eq:thin2}
holds in the quantitative form
$\lambda_\pm=\frac{\tensionsym}{R^2\delta}\bigl(\mu_\pm(k)+O(\delta)\bigr)$.
\end{remark}

\section{The thin-interlayer limit}\label{sec:thinlimit}

Throughout this section, we fix $k\ge1$ and consider the
frozen family of concentric geometries
\begin{equation}\label{eq:thin-setting}
 R_1=R,\qquad R_2=R+\delta,\qquad R_3=S=\tau R,\qquad \delta\downarrow0,
\end{equation}
with $R>0$ and $\tau>1$ fixed.  The matrix
$M_k^\delta:=M_k(R,R+\delta,S)$ of Proposition~\ref{prop:sixtrace} is
the mode matrix at these radii; it describes the instantaneous
linearization at a fixed geometry and not the evolution along the
concentric orbit.  By \eqref{eq:base-flow} the base coefficient
converges, $\alpha\to\alpha_c$, with $\alpha_c$ as in
\eqref{eq:alphac}, i.e.
$\alpha_c=((\tension1+\tension2)/R+\tension3/S)/(2\log(S/R))$.
This fixed-$k$ limit has $k\delta/R\to0$; no estimate
uniform in $k$ is claimed.

\subsection{Resistor form of the thin-annulus map}

\begin{lemma}[Resistor form]\label{lem:resistor}
Let $L_1=\log(R_2/R_1)$, $s_1=\sinh kL_1$, and let
\begin{equation}\label{eq:D}
 \Delta_1:=\bvp u_2^--\bvp u_1^+
 =\bvp u_2^--\bvp u_1^-+a_1J_1\in\R^2
\end{equation}
be the difference between the Fourier boundary coefficients at the
two ends of the inner annulus.  Then, exactly,
\begin{equation}\label{eq:resistor}
 \bvp v_1^+=\frac{k}{R_1s_1}\,\Delta_1
   -\frac{k}{R_1}\tanh\frac{kL_1}{2}\,\bvp u_1^+,
 \qquad
 \bvp v_2^-=\frac{k}{R_2s_1}\,\Delta_1
   +\frac{k}{R_2}\tanh\frac{kL_1}{2}\,\bvp u_2^- .
\end{equation}
\end{lemma}

\begin{proof}
Lemma~\ref{lem:dtn} applied to the inner annulus with boundary
coefficients $(\bvp u_1^+,\bvp u_2^-)$ gives
$\bvp v_1^+=\frac{k}{R_1s_1}(\bvp u_2^--c_1\bvp u_1^+)$ and
$\bvp v_2^-=\frac{k}{R_2s_1}(c_1\bvp u_2^--\bvp u_1^+)$,
where $c_1=\cosh(kL_1)$.  Substitute
$\bvp u_2^--c_1\bvp u_1^+
=\Delta_1-(c_1-1)\bvp u_1^+$ and
$c_1\bvp u_2^--\bvp u_1^+
=\Delta_1+(c_1-1)\bvp u_2^-$, then use
$(\cosh z-1)/\sinh z=\tanh(z/2)$.
\end{proof}

The first term in each identity carries the thin-annulus factor
$k/(R_1s_1)=\delta^{-1}(1+O(\delta))$ and multiplies the difference
$\Delta_1$ of the two boundary Fourier coefficients.  The remaining
terms are regular since
$\tanh(kL_1/2)=\frac{k\delta}{2R}(1+O(\delta))$.
All orders in $\delta$ below refer to this decomposition.

\subsection{Regularized solvability and the finite limit}

\begin{lemma}[Regularized six-coefficient system]\label{lem:L4A}
Set $X:=\bvp u_2^-$, $Y:=\bvp u_3^-$, and
$U:=\Delta_1/\delta=(\bvp u_2^--\bvp u_1^+)/\delta$;
each is in $\R^2$.  In these variables the six-coefficient system
\eqref{eq:KB} extends analytically to $\delta=0$, and its limiting
coefficient matrix is invertible.  Consequently, the boundary Fourier
coefficients, their radial-derivative coefficients, and the velocity
rows are analytic functions of $\delta$ at $\delta=0$.
\end{lemma}

\begin{proof}
Write $R_2=R+\delta$, $L_1=\log(R_2/R)$, $L_2=\log(S/R_2)$,
$e_1=(R_2/R)^k$, $s_2=\sinh kL_2$, $c_2=\cosh kL_2$, and
$\alpha=\alpha(\delta)$ for the base coefficient \eqref{eq:base-flow} at
the radii $(R,R_2,S)$.  The change of variables is exact:
$\bvp u_1^+=X-\delta U$ and $\bvp u_1^-=X+a_1J_1-\delta U$.

\pstep{1}{Coefficients}
Since $\sinh kL_1=(e_1^2-1)/(2e_1)$ and
$\tanh(kL_1/2)=(e_1-1)/(e_1+1)$, the identities \eqref{eq:resistor}
read
\[
 \bvp v_1^+=(A_1+\delta T_1)\,U-T_1X,\qquad
 \bvp v_2^-=B_1\,U+T_2X,
\]
with
\begin{equation*}
  \begin{aligned}
 A_1 &:=\frac{k\delta}{R\sinh kL_1}=\frac{2k\,e_1}{W(\delta)},\quad
 B_1:=\frac{R}{R_2}A_1,\\
 T_1 &:=\frac kR\,\frac{e_1-1}{e_1+1},\quad
 T_2:=\frac{k}{R_2}\,\frac{e_1-1}{e_1+1},\quad
 W(\delta) :=\sum_{j=0}^{2k-1}\binom{2k}{j+1}\left(\frac{\delta}{R}\right)^{j},
  \end{aligned}
\end{equation*}
where $e_1^2-1=(\delta/R)\,W(\delta)$ is the binomial expansion.
For fixed $k$ all of $e_1$, $W$, $A_1$, $B_1$, $T_1$, $T_2$, $e_2:=(S/R_2)^k$,
$s_2=(e_2^2-1)/(2e_2)$, $c_2=(e_2^2+1)/(2e_2)$ and $\alpha$ are rational
functions of $\delta$ whose denominators do not vanish on
$0\le\delta<S-R$ ($W(0)=2k$, $e_1\ge1$, $e_2>1$), with
$A_1(0)=B_1(0)=1$, $T_1(0)=T_2(0)=0$, $\alpha(0)=\alpha_c$.

\pstep{2}{The system in $(X,Y,U)$}
Substituting $\bvp u_1^\pm$, \eqref{eq:resistor}, the outer Dirichlet-to-Neumann
formulas \eqref{eq:vtraces}, and the base data of
Definition~\ref{def:traces} into the six rows of \eqref{eq:KB} gives, in
the order GT$_1$, FX$_1$, GT$_2$, FX$_2$, GT$_3$, FX$_3$,
\begin{align*}
 -2X_1-X_2+\delta(2U_1+U_2)
   &=\Bigl(\frac{q_1}{R^2}+\frac{2\alpha}{R}\Bigr)a_1,\\
 -\Bigl(T_1+\frac kR\Bigr)X_2+\Bigl(A_1+\delta T_1+\frac{k\delta}{R}\Bigr)U_2
   &=0,\\
 X_1-X_2&=\Bigl(\frac{q_2}{R_2^2}-\frac{\alpha}{R_2}\Bigr)a_2,\\
 \Bigl(\frac{kc_2}{R_2s_2}+T_2\Bigr)(X_1+X_2)-\frac{k}{R_2s_2}(Y_1+Y_2)+B_1(U_1+U_2)
   &=0,\\
 Y_1+2Y_2&=\Bigl(\frac{q_3}{S^2}-\frac{2\alpha}{S}\Bigr)a_3,\\
 \frac{k}{Ss_2}X_1-\frac kS\Bigl(1+\frac{c_2}{s_2}\Bigr)Y_1
   &=-\frac{k\alpha}{Ss_2R_2}\,a_2,
\end{align*}
with $q_j=\tension{j}(k^2-1)$; here $c_2^\top J_1=c_3^\top J_2=0$,
$c_2^\top H_1=c_3^\top H_2=c_1^\top H_3=0$, $d_1^\top J_1=-2\alpha/R$,
$d_2^\top G_2=\alpha/R_2$, $d_3^\top G_3=2\alpha/S$ and
$c_1^\top J_2=-\alpha/R_2$ were used.  The velocity rows \eqref{eq:DE}
become linear in $(X,Y,U,\bvp a)$ with coefficients from the same list.
Thus, \eqref{eq:KB} is equivalent to $\mathcal K^\delta(X,Y,U)^\top=\mathcal P^\delta\bvp a$
with $\mathcal K^\delta$, $\mathcal P^\delta$ rational in $\delta$ on
$[0,S-R)$, and the correspondence $(X,Y,U)\mapsto\bvp u^-$ is a
bijection for every $\delta>0$.

\pstep{3}{Invertibility of $\mathcal K^\delta$ for $0\le\delta<S-R$}
For $\delta>0$ the kernel of $\mathcal K^\delta$ maps injectively into the
kernel of $\mathcal K_k$, which is trivial by
Lemma~\ref{lem:global-invertibility}.  At $\delta=0$ the system reads
$-2X_1-X_2=Aa_1$, $-\frac kRX_2+U_2=0$, $X_1-X_2=Ba_2$,
$\frac{kc}{Rs}(X_1+X_2)-\frac{k}{Rs}(Y_1+Y_2)+U_1+U_2=0$,
$Y_1+2Y_2=(q_3/S^2-2\alpha_c/S)a_3$ and
$\frac{k}{Ss}X_1-\frac kS(1+\frac cs)Y_1=-\frac{k\alpha_c}{SsR}a_2$,
with $s=\sinh(k\log(S/R))$, $c=\cosh(k\log(S/R))$ and $A,B$ as in
\eqref{eq:AB}.  Solving in the order GT$_1$,GT$_2$ (determinant $3$) for
$X$, then FX$_3$,GT$_3$ for $Y$ (the coefficient $\frac kS(1+c/s)$ of $Y_1$
is positive), then FX$_1$ for $U_2$ and FX$_2$ for $U_1$ determines
$(X,Y,U)$ uniquely for every $a$; the explicit solution is
\eqref{eq:limit-X}--\eqref{eq:limit-Y} of Lemma~\ref{lem:L4B}.  Hence, it follows that
$\det\mathcal K^\delta\ne0$ on $[0,S-R)$.

\pstep{4}{Conclusion}
By Cramer's rule, the entries of $(\mathcal K^\delta)^{-1}$ are rational
functions of $\delta$ without poles on $[0,S-R)$, since the entries of
$\mathcal K^\delta$ are and $\det\mathcal K^\delta$ does not vanish there.
Hence, $(X,Y,U)=(\mathcal K^\delta)^{-1}\mathcal P^\delta\bvp a$, the derivative
coefficients $\bvp v_j^\pm$, and the velocity rows
$M_k^\delta\bvp a=\mathcal D^\delta(\mathcal K^\delta)^{-1}\mathcal P^\delta\bvp a+\mathcal E^\delta\bvp a$
are rational functions of $\delta$ without poles on $[0,S-R)$; in
particular they are analytic at $\delta=0$, and their values at
$\delta=0$ are those of the limiting system.
\end{proof}

\begin{corollary}[Finite limit]\label{cor:T1}
The mode matrix $M_k^\delta$ extends to an analytic function of $\delta$ at $\delta=0$.
In particular $M_k^\delta\to M_k^0$ entrywise and
$\delta M_k^\delta\to0$: the three-phase thin-interlayer limit has no
$O(\delta^{-1})$ singularity, in contrast with \eqref{eq:thin2}.
\end{corollary}

\begin{proof}
By Lemma~\ref{lem:L4A} the velocity rows \eqref{eq:DE}, and hence the
entries of $M_k^\delta$, are analytic at $\delta=0$.
\end{proof}

\subsection{The limit system and the gap row}

\begin{lemma}[Limit boundary coefficients, velocity rows, and gap row]\label{lem:L4B}
Let $q_j:=\tension{j}(k^2-1)$, $\alpha:=\alpha_c$,
$s=\sinh(k\log(S/R))$, $c=\cosh(k\log(S/R))$, and
\begin{equation}\label{eq:AB}
 A:=\frac{q_1}{R^2}+\frac{2\alpha}{R},\qquad
 B:=\frac{q_2}{R^2}-\frac{\alpha}{R}.
\end{equation}
At $\delta=0$, the boundary Fourier coefficient vectors
$X=\bvp u_2^-$ and $Y=\bvp u_3^-$ are
\begin{equation}\label{eq:limit-X}
 X_1=\frac{Ba_2-Aa_1}{3},\qquad
 X_2=-\frac{Aa_1+2Ba_2}{3},
\end{equation}
\begin{equation}\label{eq:limit-Y}
 Z:=X-a_2J_2,\qquad
 Y_1=e^{-k\log(S/R)}Z_1,\qquad
 Y_2=\frac12\left[\left(\frac {q_3}{S^2}-\frac{2\alpha}{S}\right)a_3-Y_1\right].
\end{equation}
With $v:=\frac{k}{Rs}(Y-cZ)$ and $w:=\frac{k}{Ss}(cY-Z)$, the three
velocity rows of $M_k^0$ are
\begin{align}
 (M_k^0\bvp a)_1&=-(v_1+v_2)+\frac{k}{R}
       \left(X_1+X_2+\frac{\alpha}{R}a_1\right)
       +\frac{\alpha}{R^2}a_1,\label{eq:limit-vel1}\\
 (M_k^0\bvp a)_2&=\frac{k}{R}X_2-v_2+\frac{\alpha}{R^2}a_2,\label{eq:limit-vel2}\\
 (M_k^0\bvp a)_3&=-\frac{k}{S}Y_2-w_2
       -\frac{(k-1)\alpha}{S^2}a_3,\label{eq:limit-vel3}
\end{align}
and the gap row is
\begin{equation}\label{eq:gap-forcing}
 (M_k^0\bvp a)_2-(M_k^0\bvp a)_1=\lambda_k^0\,g+F_k\,m
\end{equation}
with $F_k$ and $\lambda_k^0$ as in \eqref{eq:main-gap}.
\end{lemma}

\begin{proof}
At $\delta=0$ the system of Lemma~\ref{lem:L4A} reads, with
$\bvp u_1^-=X+a_1J_1$, $\bvp u_1^+=X$, $\bvp v_1^+=\bvp v_2^-=U$,
$R_2=R$ and $\alpha=\alpha_c$ in the base data,
\begin{alignat*}{3}
 &\text{GT}_1:\ &-2X_1-X_2&=Aa_1, &\qquad
 &\text{FX}_1:\ -\tfrac kRX_2+U_2=0,\\
 &\text{GT}_2:\ &X_1-X_2&=Ba_2, &
 &\text{FX}_2:\ -(v_1+v_2)+U_1+U_2=0,\\
 &\text{GT}_3:\ &Y_1+2Y_2&=\Bigl(\tfrac{q_3}{S^2}-\tfrac{2\alpha}{S}\Bigr)a_3, &
 &\text{FX}_3:\ \tfrac{k}{Ss}Z_1-\tfrac kS\Bigl(1+\tfrac cs\Bigr)Y_1=0,
\end{alignat*}
where $d_1^\top J_1=-2\alpha/R$, $G_1=0$, $d_2^\top G_2=\alpha/R$,
$d_3^\top G_3=2\alpha/S$, $c_2^\top J_1=c_3^\top J_2=0$ and
$c_2^\top H_1=c_3^\top H_2=c_1^\top H_3=0$ were used, $Z=X-a_2J_2$ with
$Z_1=X_1+\alpha a_2/R$ and $Z_1+Z_2=X_1+X_2$, and
$v=\bvp v_2^+$, $w=\bvp v_3^-$ are the outer-annulus formulas
\eqref{eq:vtraces} while $\bvp v_3^+=-\frac kS(Y-a_3J_3)$.

\pstep{1}{Boundary coefficients}
GT$_1$ and GT$_2$ form a $2\times2$ system in $X$ of determinant $3$;
its solution is \eqref{eq:limit-X}.  FX$_3$ alone gives
$(c+s)Y_1=Z_1$, i.e.\ $Y_1=(c-s)Z_1=e^{-k\log(S/R)}Z_1$ because
$(c+s)(c-s)=1$; GT$_3$ then gives $Y_2$.  This is \eqref{eq:limit-Y}.
FX$_1$ gives $U_2=\frac kRX_2$ and FX$_2$ gives $U_1=v_1+v_2-U_2$.

\pstep{2}{Velocity rows}
Substituting $U$ into \eqref{eq:DE} with
$c_3^\top J_1=-\alpha/R$, $c_3^\top H_1=c_1^\top H_2=\alpha/R^2$,
$c_2^\top J_3=-\alpha/S$ and $c_2^\top H_3=\alpha/S^2$ gives
\eqref{eq:limit-vel1}--\eqref{eq:limit-vel3}.

\pstep{3}{Gap row}
Since $Y_1=(c-s)Z_1$, one has $v_1=\frac{k}{Rs}(Y_1-cZ_1)=-\frac kRZ_1$.
Substituting this, $U_2=\frac kRX_2$ and $Z_1=X_1+\alpha a_2/R$ into
\eqref{eq:limit-vel2}--\eqref{eq:limit-vel1}, all terms in $X_2$ and
$v_2$ cancel and
\[
 (M_k^0\bvp a)_2-(M_k^0\bvp a)_1
 =-\frac{2k}{R}X_1-\frac{\alpha}{R^2}\bigl((k+1)a_1+(k-1)a_2\bigr).
\]
With $a_1=m-g/2$, $a_2=m+g/2$ and \eqref{eq:limit-X},
$3X_1=(B-A)m+\tfrac12(A+B)g$ and $(k+1)a_1+(k-1)a_2=2km-g$, where
$B-A=(q_2-q_1)/R^2-3\alpha/R$ and $A+B=(q_1+q_2)/R^2+\alpha/R$.
Hence, the coefficient of $m$ is
$-\frac{2k}{3R}(B-A)-\frac{2k\alpha}{R^2}=\frac{2k(q_1-q_2)}{3R^3}=F_k$,
the terms in $\alpha$ cancelling, and the coefficient of $g$ is
$-\frac{k}{3R}(A+B)+\frac{\alpha}{R^2}
=-\frac{k(q_1+q_2)}{3R^3}+\frac{(3-k)\alpha}{3R^2}=\lambda_k^0$.
\end{proof}

\subsection{The composite block}

Let $\Mcomp{k}$ denote the matrix \eqref{eq:W5} of the two-phase
annulus with radii $(R_1,R_2)=(R,S)$ and tensions
$\tension{\rm in}=\tension1+\tension2$, $\tension{\rm out}=\tension3$.
Its base coefficient \eqref{eq:base2} is
$\bar\alpha=((\tension1+\tension2)/R+\tension3/S)/(2\log(S/R))=\alpha_c$,
so that with $g_{\rm in}:=(q_1+q_2)/(2R^2)$ and
$g_{\rm out}:=q_3/(2S^2)$,
\begin{equation}\label{eq:Mcomp-explicit}
 \Mcomp{k}=
 \begin{pmatrix}
 -\dfrac{kc}{Rs}\Bigl(g_{\rm in}+\dfrac{\alpha_c}R\Bigr)
   -\dfrac{k g_{\rm in}}R+\dfrac{\alpha_c}{R^2}
 &-\dfrac{k}{Rs}\Bigl(g_{\rm out}-\dfrac{\alpha_c}S\Bigr)\\[3mm]
 -\dfrac{k}{Ss}\Bigl(g_{\rm in}+\dfrac{\alpha_c}R\Bigr)
 &-\dfrac{kc}{Ss}\Bigl(g_{\rm out}-\dfrac{\alpha_c}S\Bigr)
   -\dfrac{k g_{\rm out}}S+\dfrac{\alpha_c}{S^2}
 \end{pmatrix}.
\end{equation}
It is the mode matrix of a two-interface problem in which the inner
interface carries the sum of the two collapsed tensions; it is not the
mode matrix of a two-phase model with one uniform tension between
phases $3$ and $2$.

\begin{lemma}[Composite block]\label{lem:L4D}
Let $i(m,a_3):=(m,m,a_3)^\top$ and
$P(a_1,a_2,a_3):=\bigl(\tfrac12(a_1+a_2),a_3\bigr)^\top$.  For
arbitrary positive tensions,
\begin{equation}\label{eq:PMi}
 PM_k^0\,i=\Mcomp{k}.
\end{equation}
More precisely, for $\bvp a=i(m,a_3)$,
\begin{equation}\label{eq:inner-rows}
 (M_k^0\bvp a)_1=(\Mcomp{k}(m,a_3)^\top)_1-\tfrac12F_km,\qquad
 (M_k^0\bvp a)_2=(\Mcomp{k}(m,a_3)^\top)_1+\tfrac12F_km,
\end{equation}
and $(M_k^0\bvp a)_3=(\Mcomp{k}(m,a_3)^\top)_2$.
\end{lemma}

\begin{proof}
Let $a_1=a_2=m$ and keep the notation of Lemma~\ref{lem:L4B}.  Put
$d:=(q_1-q_2)/(6R^2)$ and $e:=e^{-k\log(S/R)}=c-s$.  Then
\eqref{eq:limit-X}--\eqref{eq:limit-Y} become
\[
 X=\begin{pmatrix}-(\alpha/R+2d)m\\(-g_{\rm in}+d)m\end{pmatrix},\quad
 Z=\begin{pmatrix}-2dm\\-(g_{\rm in}+\alpha/R-d)m\end{pmatrix},\quad
 Y=\begin{pmatrix}-2edm\\(g_{\rm out}-\alpha/S)a_3+edm\end{pmatrix}.
\]
Substituting into \eqref{eq:limit-vel1}--\eqref{eq:limit-vel3} gives
\begin{align*}
 (M_k^0\bvp a)_1&=(\Mcomp{k}(m,a_3)^\top)_1-\frac{2kdm}{R},&
 (M_k^0\bvp a)_2&=(\Mcomp{k}(m,a_3)^\top)_1+\frac{2kdm}{R},\\
 (M_k^0\bvp a)_3&=(\Mcomp{k}(m,a_3)^\top)_2 ,
\end{align*}
with $\Mcomp{k}$ as in \eqref{eq:Mcomp-explicit}: the terms in
$d$ cancel in the average of the first two rows and in the third row,
the latter by $e(c+s)=1$.  Since $2kd/R=F_k/2$, this is
\eqref{eq:inner-rows}, and averaging gives \eqref{eq:PMi}.
\end{proof}

\subsection{Main result and consequence}

\begin{proof}[Proof of Theorem~\ref{thm:main}]
Part (a) is Corollary~\ref{cor:T1}.  For (b), Lemma~\ref{lem:L4B}
shows that the gap row of $M_k^0$ in the coordinates $(m,g,a_3)$ is
$(F_k,\lambda_k^0,0)$, and Lemma~\ref{lem:L4D} shows that the
$(m,a_3)$-entries of the $m$-row and of the $a_3$-row are those of
$\Mcomp{k}$; this is \eqref{eq:main-form}, with $c_1$ and $c_3$
the $g$-entries of these two rows.  For (c), the mean thickness along
the concentric solution grows at the logarithmic rate $\alpha/(R_1R_2)$
by Corollary~\ref{cor:hC}, which tends to $\alpha_c/R^2$ as
$\delta\downarrow0$; subtracting it from $\lambda_k^0$ gives
\eqref{eq:main-relative}, which is negative because $k\ge1$,
$\tension1+\tension2>0$ and $\alpha_c>0$.  For (d), $F_k=0$ exactly
when $\tension1=\tension2$ or $k=1$; then the $g$-row of
\eqref{eq:main-form} is $(0,\lambda_k^0,0)$, so $\dot g=\lambda_k^0g$
and $V=\{g=0\}$ is invariant, the restriction of $M_k^0$ to
$V\cong\{(m,a_3)\}$ is $\Mcomp{k}$, and expanding
$\det(\mu I-M_k^0)$ along the $g$-row gives
$(\mu-\lambda_k^0)\det(\mu I-\Mcomp{k})$, which is
\eqref{eq:main-split}.
\end{proof}

\begin{corollary}[Finite thickness]\label{cor:T3}
For fixed $k$, the coefficients of the gap row \eqref{eq:gaprow} of
$M_k^\delta$ satisfy, as $\delta\downarrow0$,
\begin{equation}\label{eq:beta-orders}
 \beta_m=F_k+O(\delta),\qquad
 \beta_g=\lambda_k^0+O(\delta),\qquad
 \beta_3=O(\delta).
\end{equation}
Hence, $\Lambda_k(R,R+\delta,S)=\Lambda_k^0+O(\delta)$, so the
relative gap rate is negative for all sufficiently
small $\delta$ (depending on $k$), and
\begin{equation}\label{eq:scalarized}
 \dot g=\lambda_k^0\,g+F_k\,m
 +O(\delta)\bigl(|m|+|g|+|a_3|\bigr).
\end{equation}
The gap row is autonomous in the limit exactly when $F_k=0$.
\end{corollary}

\begin{proof}
By Lemma~\ref{lem:L4A} each coefficient is analytic in $\delta$ at
$\delta=0$, and by Lemma~\ref{lem:L4B} its limit is the corresponding
entry of $(F_k,\lambda_k^0,0)$; moreover
$\alpha/(R_1R_2)=\alpha_c/R^2+O(\delta)$.
\end{proof}

\begin{corollary}[Sign of the pure-gap rate]\label{cor:T5}
For all positive tensions, $\lambda_1^0=\tfrac23\alpha_c/R^2>0$ and
$\lambda_k^0<0$ for every $k\ge3$.  For $k=2$, with
$\varrho:=\tension3/(\tension1+\tension2)$,
\begin{equation}\label{eq:threshold}
 \lambda_2^0>0
 \iff \frac{1+\varrho/\tau}{2\log\tau}>6
 \iff \tau<\tau_*(\varrho),
\end{equation}
where $\tau_*(\varrho)$ is the unique root of the middle expression
equal to $6$; for equal tensions ($\varrho=\tfrac12$) the condition
reads $(2\tau+1)/(2\tau\log\tau)>12$ and $\tau_*=1.1278104\ldots$.
For every $k\ge1$, $\Lambda_k^0<0$.  When
$\tau<\tau_*(\varrho)$, the instantaneous $k=2$ pure-gap rate
$\lambda_2^0$ is positive, while the relative gap rate $\Lambda_2^0$ is
negative.
\end{corollary}

\begin{proof}
Formula \eqref{eq:main-gap} gives $\lambda_1^0=\tfrac23\alpha_c/R^2>0$.
For $k\ge3$ its first term is negative, and its second one is nonpositive, so
$\lambda_k^0<0$.  For $k=2$ it reads
$\lambda_2^0=-2(\tension1+\tension2)/R^3+\alpha_c/(3R^2)$; with
$\varphi(\tau,\varrho):=(1+\varrho/\tau)/(2\log\tau)$ one has
$\alpha_cR=(\tension1+\tension2)\varphi(\tau,\varrho)$ by \eqref{eq:alphac},
hence
\[
 \lambda_2^0=\frac{\tension1+\tension2}{3R^3}\bigl(\varphi(\tau,\varrho)-6\bigr),
\]
which gives the first equivalence in \eqref{eq:threshold}.  For fixed
$\varrho>0$, $\varphi(\cdot,\varrho)$ is continuous and strictly decreasing
on $(1,\infty)$ (its numerator is positive and decreasing, its denominator
positive and increasing), with $\varphi(1.01,\varrho)>6$ (since
$\log\tau\le\tau-1$) and $\varphi(e^{1+\varrho},\varrho)<\tfrac12$.  By the
intermediate value theorem there is exactly one $\tau_*(\varrho)>1$ with
$\varphi(\tau_*,\varrho)=6$, and $\varphi(\tau,\varrho)>6$ if and only if
$\tau<\tau_*(\varrho)$; this is the second equivalence.  For
$\varrho=\tfrac12$ the condition $\varphi>6$ reads
$(2\tau+1)/(2\tau\log\tau)>12$, and $1.1278104<\tau_*<1.1278105$ follows
from $\varphi(1.1278104,\tfrac12)>6>\varphi(1.1278105,\tfrac12)$, where
$\log\tau$ is bracketed by the twelfth partial sum of
$-\log(1-x)=\sum_{i\ge1}x^i/i$ at $x=1-1/\tau$ together with the
geometric bound $x^{13}/(1-x)<6\cdot10^{-13}$ on its remainder, which is
far below the accuracy $2.4\cdot10^{-8}$ in $\log\tau$ that separating
the two endpoints requires; in particular the seven digits displayed in
$\tau_*=1.1278104\ldots$ are correct.  Finally, by \eqref{eq:main-gap} and \eqref{eq:main-relative},
\[
 \Lambda_k^0=-\frac{k(k^2-1)(\tension1+\tension2)}{3R^3}-\frac{k\,\alpha_c}{3R^2}<0
 \qquad(k\ge1),
\]
both terms being nonpositive and the second negative.
\end{proof}

\begin{remark}[Location of the frozen positive mode]
\label{rem:positive-comp}
For equal unit tensions, $R=1$, $S=2$ and $k=2$, the composite matrix
\eqref{eq:Mcomp-explicit} has $\alpha_c=5/(4\log2)$ and
\[
 \det \Mcomp{2}
 =\frac{48}{5}-\frac{44}{5}\alpha_c-\frac34\alpha_c^2<0 ,
\]
because $\log2<1$ gives $\alpha_c>5/4$ and the right side is negative
already at $\alpha_c=5/4$.  Its eigenvalues are therefore real and of
opposite signs; the positive one is numerically $0.5809$.  By
Theorem~\ref{thm:main}(d),
$\operatorname{spec}M_2^0=\{\lambda_2^0\}\cup\operatorname{spec}\Mcomp{2}$
with $\lambda_2^0=-4+\alpha_c/3<0$, so the limiting matrix has exactly
one positive eigenvalue, and its eigenvector lies in $V$: it is a mode
of the outer shell with vanishing thickness modulation.  For small
$\delta>0$ the simple positive root persists by continuity, with a gap
component of order $\delta$ in a normalized eigenvector.  A positive frozen
eigenvalue is therefore not in conflict with the negative relative
gap rate \eqref{eq:main-relative}: in the thin-layer limit it
belongs to the composite motion.
\end{remark}

\section*{Formal verification and generative AI use}
Selected results for the fixed-mode linearized system were formalized in Lean 4.
Claude Code (Anthropic) together with Codex and ChatGPT (OpenAI) assisted with drafting and editing the manuscript and in developing the Lean formalization.
The author reviewed the resulting text, proofs, citations, and formalization and takes responsibility for the final work.

\bibliographystyle{plainnat}
\bibliography{paper4_draft}

\end{document}